\documentclass[11pt]{amsart}%
\usepackage{graphicx}
\usepackage{amscd, color}
\usepackage{amsmath}
\usepackage{amsfonts}
\usepackage{amssymb}%
\providecommand{\U}[1]{\protect\rule{.1in}{.1in}}
\providecommand{\U}[1]{\protect\rule{.1in}{.1in}}
\providecommand{\U}[1]{\protect\rule{.1in}{.1in}} \textwidth 15.8cm
\theoremstyle{plain}

\newtheorem{theorem}{Theorem}[section]
\newtheorem{proposition}[theorem]{Proposition}

\newtheorem{example}[theorem]{Example}

\newtheorem{remark}[theorem]{Remark}

\newtheorem{definition}[theorem]{Definition}
\numberwithin{equation}{section}

\begin{document}
\title[On multi-ideals and polynomial ideals of Banach spaces]{On multi-ideals and polynomial ideals of
Banach spaces: a new approach to coherence and compatibility}
\author{Daniel Pellegrino and Joilson Ribeiro}

\address{Departamento de Matem\'{a}tica,\newline\indent Universidade Federal da Para\'{\i}ba,\newline\indent 58.051-900 - Jo\~{a}o Pessoa, Brazil.}
\email{dmpellegrino@gmail.com and pellegrino@pq.cnpq.br}

\address{Instituto de Matem\'{a}tica, \newline\indent Universidade Federal da Bahia, \newline\indent Av. Adhemar de Barros, s/n, Sala 268, \newline\indent Salvador, 40170110, Brazil.}
\email{joilsonribeiro@yahoo.com.br}

\begin{abstract}
What is an adequate extension of an operator ideal $\mathcal{I}$ to
the polynomial and multilinear settings? This question motivated the appearance of the concepts of coherent sequences of
polynomial ideals and compatibility of a polynomial ideal with an operator
ideal, introduced by D. Carando \textit{el al.} We propose a different approach by considering
pairs $(\mathcal{U}_{k},\mathcal{M}_{k})_{k=1}^{\infty}$, where $(\mathcal{U}%
_{k})_{k=1}^{\infty}$ is a polynomial ideal and $(\mathcal{M}_{k}%
)_{k=1}^{\infty}$ is a multi-ideal, instead of considering just polynomial
ideals. Our approach ends a discomfort caused by the previous theory: for real scalars the canonical sequence
$(\mathcal{P}_{k})_{k=1}^{\infty}$ of continuous $k$-homogeneous polynomials is not coherent according to the definition
of Carando \textit{et al.}

We apply these new notions to test the factorization method and different
classes that generalise the concept of absolutely summing operator.

\end{abstract}
\maketitle


\section{Introduction and background}

The origin of the theory of operator ideals goes back to 1941, with J.W.
Calkin \cite{cal} and subsequent works of H. Weyl \cite{we} and
A.\ Grothendieck \cite{grote} but only in the 70's, with the work of A.
Pietsch \cite{mono}, the basic concepts were organised and presented (see also
\cite{DJP, HP}). For applications of the subject in different areas of
mathematics we refer the reader to the recent paper \cite{DJP} and for more
historical details we refer to \cite{hist}. The extension to multilinear
mappings, with the concept of multi-ideals, is also due to Pietsch \cite{PPPP}.

The investigation of special sets of homogeneous polynomials and multilinear
mappings between Banach spaces is historically motivated by two different
directions: infinite-dimensional holomorphy (\cite{din, Mu}) or the theory of
operator ideals, polynomial ideals and multi-ideals (\cite{mono, PPPP}). The
holomorphic approach, of course, is mostly focused on polynomials rather than
on multilinear operators (\cite{CDM09, CDM33, Nachbin}). In this paper we will
be mainly motivated by the theory of operator ideals and for this reason we
are also interested in multi-ideals.

Several common multi-ideals and polynomial ideals are usually associated to
some operator ideal; however, the extension of an operator ideal to
polynomials and multilinear mappings is not always a simple task. For example,
the ideal of absolutely summing operators has, at least, eight
possible\ extensions to higher degrees (see, for example, \cite{comp, rm,
dimant, Collect, Matos-N, QM, davidarchiv} and references therein).

The main goal of this paper is to introduce clear general criteria to decide
how multilinear and polynomial generalisations of a given operator ideal must
behave in order to be considered \textquotedblleft adequate\textquotedblright.
Our criteria are defined simultaneously to pairs of ideals of polynomials and
multi-ideals and, to the best of our knowledge, this is the first attempt in
this direction in the literature. The arguments used throughout this paper are
fairly clear and simple in nature but we do believe that they can shed some
light to future works on ideals of operators, polynomials and multilinear mappings.

From now on the letters $E,E_{1},...,E_{n},F,G,H$ will represent Banach spaces
over the same scalar-field $\mathbb{K}=\mathbb{R}$ or $\mathbb{C}$.

An operator ideal $\mathcal{I}$ is a subclass of the class $\mathcal{L}_{1}$
of all continuous linear operators between Banach spaces such that for all
Banach spaces $E$ and $F$ its components%
\[
\mathcal{I}(E;F):=\mathcal{L}_{1}(E;F)\cap\mathcal{I}%
\]
satisfy:

(Oa) $\mathcal{I}(E;F)$ is a linear subspace of $\mathcal{L}_{1}(E;F)$ which
contains the finite rank operators.

(Ob) If $u\in\mathcal{I}(E;F)$, $v\in\mathcal{L}_{1}(G;E)$ and $w\in
\mathcal{L}_{1}(F;H)$, then $w\circ u\circ v\in\mathcal{I}(G;H)$.

The operator ideal is a normed operator ideal if there is a function
$\Vert\cdot\Vert_{\mathcal{I}}\colon\mathcal{I}\longrightarrow\lbrack
0,\infty)$ satisfying

\bigskip

(O1) $\Vert\cdot\Vert_{\mathcal{I}}$ restricted to $\mathcal{I}(E;F)$ is a
norm, for all Banach spaces $E$, $F$.

(O2) $\Vert P_{1}\colon\mathbb{K}\longrightarrow\mathbb{K}:P_{1}%
(\lambda)=\lambda\Vert_{\mathcal{I}}=1.$

(O3) If $u\in\mathcal{I}(E;F)$, $v\in\mathcal{L}_{1}(G;E)$ and $w\in
\mathcal{L}_{1}(F;H)$, then
\[
\Vert w\circ u\circ v\Vert_{\mathcal{I}}\leq\Vert w\Vert\Vert u\Vert
_{\mathcal{I}}\Vert v\Vert.
\]
When $\mathcal{I}(E;F)$ with the norm above is always complete, $\mathcal{I}$
is called a Banach operator ideal.

For each positive integer $n$, let $\mathcal{L}_{n}$ denote the class of all
continuous $n$-linear operators between Banach spaces. An ideal of multilinear
mappings (or multi-ideal) $\mathcal{M}$ is a subclass of the class
$\mathcal{L}=%
{\textstyle\bigcup\limits_{n=1}^{\infty}}
\mathcal{L}_{n}$ of all continuous multilinear operators between Banach spaces
such that for a positive integer $n$, Banach spaces $E_{1},\ldots,E_{n}$ and
$F$, the components
\[
\mathcal{M}_{n}(E_{1},\ldots,E_{n};F):=\mathcal{L}_{n}(E_{1},\ldots
,E_{n};F)\cap\mathcal{M}%
\]
satisfy:

\bigskip

(Ma) $\mathcal{M}_{n}(E_{1},\ldots,E_{n};F)$ is a linear subspace of
$\mathcal{L}_{n}(E_{1},\ldots,E_{n};F)$ which contains the $n$-linear mappings
of finite type.

(Mb) If $T\in\mathcal{M}_{n}(E_{1},\ldots,E_{n};F)$, $u_{j}\in\mathcal{L}%
_{1}(G_{j};E_{j})$ for $j=1,\ldots,n$ and $v\in\mathcal{L}_{1}(F;H)$, then
\[
v\circ T\circ(u_{1},\ldots,u_{n})\in\mathcal{M}_{n}(G_{1},\ldots,G_{n};H).
\]

Moreover, $\mathcal{M}$ is a (quasi-) normed multi-ideal if there is a
function $\Vert\cdot\Vert_{\mathcal{M}}\colon\mathcal{M}\longrightarrow
\lbrack0,\infty)$ satisfying

\bigskip

(M1) $\Vert\cdot\Vert_{\mathcal{M}}$ restricted to $\mathcal{M}_{n}%
(E_{1},\ldots,E_{n};F)$ is a (quasi-) norm, for all Banach spaces
$E_{1},\ldots,E_{n}$ and $F.$

(M2) $\Vert T_{n}\colon\mathbb{K}^{n}\longrightarrow\mathbb{K}:T_{n}%
(\lambda_{1},\ldots,\lambda_{n})=\lambda_{1}\cdots\lambda_{n}\Vert
_{\mathcal{M}}=1$ for all $n$,

(M3) If $T\in\mathcal{M}_{n}(E_{1},\ldots,E_{n};F)$, $u_{j}\in\mathcal{L}%
_{1}(G_{j};E_{j})$ for $j=1,\ldots,n$ and $v\in\mathcal{L}_{1}(F;H)$, then
\[
\Vert v\circ T\circ(u_{1},\ldots,u_{n})\Vert_{\mathcal{M}}\leq\Vert
v\Vert\Vert T\Vert_{\mathcal{M}}\Vert u_{1}\Vert\cdots\Vert u_{n}\Vert.
\]
When all the components $\mathcal{M}_{n}(E_{1},\ldots,E_{n};F)$ are complete
under this (quasi-) norm, $\mathcal{M}$ is called a (quasi-) Banach
multi-ideal. For a fixed multi-ideal $\mathcal{M}$ and a positive integer $n$,
the class%
\[
\mathcal{M}_{n}:=\cup_{E_{1},...,E_{n},F}\mathcal{M}_{n}\left(  E_{1}%
,...,E_{n};F\right)
\]
is called ideal of $n$-linear mappings.

Analogously, for each positive integer $n$, let $\mathcal{P}_{n}$ denote the
class of all continuous $n$-homogeneous polynomials between Banach spaces. A
polynomial ideal $\mathcal{Q}$ is a subclass of the class $\mathcal{P}=%
{\textstyle\bigcup\limits_{n=1}^{\infty}}
\mathcal{P}_{n}$ of all continuous homogeneous polynomials between Banach
spaces so that for all $n\in\mathbb{N}$ and all Banach spaces $E$ and $F$, the
components
\[
\mathcal{Q}_{n}\left(  ^{n}E;F\right)  :=\mathcal{P}_{n}\left(  ^{n}%
E;F\right)  \cap\mathcal{Q}%
\]
satisfy:

(Pa) $\mathcal{Q}_{n}\left(  ^{n}E;F\right)  $ is a linear subspace of
$\mathcal{P}_{n}\left(  ^{n}E;F\right)  $ which contains the finite-type polynomials.

(Pb) If $u\in\mathcal{L}_{1}\left(  G;E\right)  $, $P\in\mathcal{Q}_{n}\left(
^{n}E;F\right)  $ and $w\in\mathcal{L}_{1}\left(  F;H\right)  $, then
\[
w\circ P\circ u\in\mathcal{Q}_{n}\left(  ^{n}G;H\right)  .
\]
\bigskip

If there exists a map $\left\Vert \cdot\right\Vert _{\mathcal{Q}}%
:\mathcal{Q}\rightarrow\lbrack0,\infty\lbrack$ satisfying

(P1) $\left\Vert \cdot\right\Vert _{\mathcal{Q}}$ restricted to $\mathcal{Q}%
_{n}(^{n}E;F)$ is a (quasi-) norm for all Banach spaces $E$ and $F$ and all
$n$;

(P2) $\left\Vert P_{n}:\mathbb{K}\rightarrow\mathbb{K};\text{ }P_{n}\left(
\lambda\right)  =\lambda^{n}\right\Vert _{\mathcal{Q}}=1$ for all $n$;

(P3) If $u\in\mathcal{L}_{1}(G;E)$, $P\in\mathcal{Q}_{n}(^{n}E;F)$ and
$w\in\mathcal{L}_{1}(F;H),$ then
\[
\left\Vert w\circ P\circ u\right\Vert _{\mathcal{Q}}\leq\left\Vert
w\right\Vert \left\Vert P\right\Vert _{\mathcal{Q}}\left\Vert u\right\Vert
^{n},
\]
$\mathcal{Q}$ is called (quasi-) normed polynomial ideal. If all components
$\mathcal{Q}_{n}\left(  ^{n}E;F\right)  $ are complete, $\left(
\mathcal{Q},\left\Vert \cdot\right\Vert _{\mathcal{Q}}\right)  $ is called a
(quasi-) Banach ideal of polynomials (or (quasi-) Banach polynomial ideal).
For a fixed ideal of polynomials $\mathcal{Q}$ and $n\in\mathbb{N}$, the
class
\[
\mathcal{Q}_{n}:=\cup_{E,F}\mathcal{Q}_{n}\left(  ^{n}E;F\right)
\]
is called ideal of $n$-homogeneous polynomials.

A relevant question in the multilinear/polynomial setting is: given an
operator ideal (for example the class of compact operators, weakly compact
operators, absolutely summing operators, strictly singular operators, etc) how
to define a multi-ideal and a polynomial ideal related to the linear operator
ideal without loosing its essence? How to lift the core of an operator ideal
to polynomials and multilinear mappings?

The crucial point is that in general a given operator ideal has several
different possible extensions to the setting of multi-ideals and polynomial
ideals (see, for example the case of absolutely summing operators \cite{rm,
advances, davidarchiv} and almost summing operators \cite{BBJ-Archiv, JMAA,
joilson}). In order to have a method of evaluating what kind of
multilinear/polynomial extensions of a given operator ideal is more adequate
and less artificial, some efforts have been done in the last years, but these
efforts were not addressed simultaneously to polynomials and multilinear mappings.

In this paper we intend to contribute in the direction of identifying the
precise properties that polynomial and multi-ideals must fulfill
(simultaneously) in order to keep the essence of a given operator ideal. The
paper is organised as follows:

In Section \ref{yyre} we recall the concepts already created for this task
(ideals closed for scalar multiplication, ideals closed under differentiation,
holomorphy types, coherent ideals, compatible ideals).

In Section \ref{yu} we present our new approach to coherence and compatibility
of pairs and in the subsequent sections we test our approaches to the
factorisation method and for pairs of ideals that generalise the ideal of
absolutely summing operators. In the last section we sketch a stronger notion
of coherence and compatibility.

\section{Coherent polynomial ideals, global holomorphy types and related
concepts\label{yyre}}

In this section we recall old and recent concepts that, in some sense,
evaluate how good a multilinear/polynomial extension of an operator ideal is.
The concepts of ideals of polynomials closed under differentiation and closed
for scalar multiplication were introduced in \cite{indagationes} (see also
\cite{BBJP} for related notions) as an attempt of identifying crucial
properties that polynomial ideals must verify in order to maintain some
harmony between the different levels of homogeneity.

From now on $E^{\ast}$ denotes the topological dual of $E$ and we use the
following notation:

\begin{itemize}
\item If $P\in\mathcal{P}_{n}\left(  ^{n}E;F\right)  $, then $\overset{\vee
}{P}$ denotes the unique symmetric $n$-linear mapping associated to $P$.

\item If $P\in\mathcal{P}_{n}\left(  ^{n}E;F\right)  $, then $P_{a^{k}}%
\in\mathcal{P}_{n-k}\left(  ^{n-k}E;F\right)  $ is defined by%
\[
P_{a^{k}}(x):=\overset{\vee}{P}(a,...,a,x,...,x).
\]

\item If $T\in\mathcal{L}_{n}\left(  E_{1},\ldots,E_{n};F\right)  $, then
$T_{a_{1},...,a_{k-1}a_{k+1}...a_{n}}\in\mathcal{L}_{1}\left(  E_{k};F\right)
$ denotes the mapping%
\[
T_{a_{1},...,a_{k-1}a_{k+1}...a_{n}}(x_{k}):=T(a_{1},...,a_{k-1},x_{k}%
,a_{k+1},...,a_{n}).
\]

\item If $T\in\mathcal{L}_{n}\left(  E_{1},\ldots,E_{n};F\right)  $, then
$T_{a_{j}}\in\mathcal{L}_{n-1}\left(  E_{1},...,E_{j-1},E_{j+1},...,E_{n}%
;F\right)  $ is given by%
\[
T_{a_{j}}(x_{1},...,x_{j-1},x_{j+1},...,x_{n}):=T(x_{1},...,x_{j-1}%
,a_{j},x_{j+1},...,x_{n}).
\]

\end{itemize}

\begin{definition}
[csm and cud polynomial ideals \cite{indagationes}]Let $\mathcal{Q}$ be a
polynomial ideal, $n\in\mathbb{N}$, $E$ and $F$ Banach spaces.

$\left(  \mathbf{i}\right)  $ $\mathcal{Q}$ is closed under differentiation
(cud) for $n$, $E$ and $F$ if $\hat{d}P\left(  a\right)  \in\mathcal{Q}%
_{1}\left(  E;F\right)  $ for all $a\in E$ and $P\in\mathcal{Q}_{n}\left(
^{n}E;F\right)  $, where $\hat{d}P\left(  a\right)  $ is the derivative of $P$
at $a$.

$\left(  \mathbf{ii}\right)  $ $\mathcal{Q}$ is closed for scalar
multiplication (csm) for $n$, $E$ and $F$ if $\varphi P\in\mathcal{Q}%
_{n+1}\left(  ^{n+1}E;F\right)  $ for all $\varphi\in E^{\ast}$ and
$P\in\mathcal{Q}_{n}\left(  ^{n}E;F\right)  .$

When $\left(  \mathbf{i}\right)  $ and/or $\left(  \mathbf{ii}\right)  $ holds
true for all $n$, $E$ and $F$ we say that $\mathcal{Q}$ is cud and/or csm.
\end{definition}

The same concept can be translated, \textit{mutatis mutandis}, to multilinear mappings.

Recently, in \cite{CDM09} (for related papers see also \cite{CDM33, CDM}), the
interesting notions of compatible polynomial ideals and coherent ideals were
presented with the same aim of filtering good polynomial extensions of given
operator ideals. This approach (with its nice self-explanatory terminology)
offers, with the notion of compatibility, a clear proposal of classifying when
a polynomial ideal has the spirit of a given operator ideal (our notation
essentially follows \cite{CDM09}):

\begin{definition}
[Compatible polynomial ideals \cite{CDM09}]\label{IdeaisCompativeis} Let
$\mathcal{U}$ be a normed ideal of linear operators. The normed ideal of
$n$-homogeneous polynomials $\mathcal{U}_{n}$ is compatible with $\mathcal{U}$
if there exist positive constants $\alpha_{1}$ and $\alpha_{2}$ such that for
all Banach spaces $E$ and $F,$ the following conditions hold:

$\left(  \mathbf{i}\right)  $ For each $P\in\mathcal{U}_{n}\left(
^{n}E;F\right)  $ and $a\in E$, $P_{a^{n-1}}$ belongs to $\mathcal{U}\left(
E;F\right)  $ and
\[
\left\Vert P_{a^{n-1}}\right\Vert _{\mathcal{U}}\leq\alpha_{1}\left\Vert
P\right\Vert _{\mathcal{U}_{n}}\left\Vert a\right\Vert ^{n-1}%
\]

$\left(  \mathbf{ii}\right)  $ For each $u\in\mathcal{U}\left(  E;F\right)  $
and $\gamma\in E^{\ast}$, $\gamma^{n-1}u$ belongs to $\mathcal{U}_{n}\left(
^{n}E;F\right)  $ and
\[
\left\Vert \gamma^{n-1}u\right\Vert _{\mathcal{U}_{n}}\leq\alpha_{2}\left\Vert
\gamma\right\Vert ^{n-1}\left\Vert u\right\Vert _{\mathcal{U}}%
\]

\end{definition}

\begin{remark}
Sometimes, for the sake of simplicity, we will say that the sequence
$(\mathcal{U}_{k})_{k=1}^{N}$ is compatible with $\mathcal{U}$ (this will mean
that $\mathcal{U}_{k}$ is compatible with $\mathcal{U}$ for all $k=1,...,N$).
\end{remark}

\begin{definition}
[Coherent polynomial ideals \cite{CDM09}]\label{IdeaisCoerentes} Consider a
sequence $(\mathcal{U}_{k})_{k=1}^{N}$, where for each $k$, $\mathcal{U}_{k}$
is an ideal of $k$-homogeneous polynomials and $N$ is eventually infinity. The
sequence $(\mathcal{U}_{k})_{k=1}^{N}$ is a coherent sequence of polynomial
ideals if there exist positive constants $\beta_{1}$ and $\beta_{2}$ such that
for all Banach spaces the following conditions hold for $k=1,...,N-1$:

$\left(  \mathbf{i}\right)  $ For each $P\in\mathcal{U}_{k+1}\left(
^{k+1}E;F\right)  $ and $a\in E$, $P_{a}$ belongs to $\mathcal{U}_{k}\left(
^{k}E;F\right)  $ and
\[
\left\Vert P_{a}\right\Vert _{\mathcal{U}_{k}}\leq\beta_{1}\left\Vert
P\right\Vert _{\mathcal{U}_{k+1}}\left\Vert a\right\Vert .
\]

$\left(  \mathbf{ii}\right)  $ For each $P\in\mathcal{U}_{k}\left(
^{k}E;F\right)  $ and $\gamma\in E^{\ast}$, $\gamma P$ belongs to
$\mathcal{U}_{k+1}\left(  ^{k+1}E;F\right)  $ and
\[
\left\Vert \gamma P\right\Vert _{\mathcal{U}_{k+1}}\leq\beta_{2}\left\Vert
\gamma\right\Vert \left\Vert P\right\Vert _{\mathcal{U}_{k}}%
\]

\end{definition}

The philosophy of the concepts above is that given positive integers $n_{1}$
and $n_{2}$, the respective levels of $n_{1}$-linearity and $n_{2}$-linearity
of a given multi-ideal (or polynomial ideal) must show some relevant
inter-connection and also keep the spirit of the original level $(n=1)$.

There is no doubt that the concepts of compatible polynomial ideals and
coherent ideals have added important contribution to the theory of polynomial
ideals. However, an operator ideal ${\mathcal{I}}$ can be always extended to
the multilinear and polynomial settings (at least in an abstract sense; for
details see \cite{not}); so, there is no apparent reason to consider the
concepts of compatibility and coherence just for polynomials (or just for
multilinear mappings separately). Our proposal offers significant variations
of these notions (as it will be clear in the next paragraph) by considering
pairs $(\mathcal{U}_{k},\mathcal{M}_{k})_{k=1}^{\infty}$, where $(\mathcal{U}%
_{k})_{k=1}^{\infty}$ is a polynomial ideal and $(\mathcal{M}_{k}%
)_{k=1}^{\infty}$ is a multi-ideal. So, this new approach deals simultaneously
with polynomials and multilinear operators and, of course, asks for some
harmony between $(\mathcal{U}_{k})_{k=1}^{\infty}$ and $(\mathcal{M}%
_{k})_{k=1}^{\infty}.$

It is very important to recall that, for the case of real scalars, the
canonical sequence $(\mathcal{P}_{k})_{k=1}^{\infty}$, composed by the ideals
of continuous $k$-homogeneous polynomials with the $\sup$ norm, is not
coherent according to Definition \ref{IdeaisCoerentes}\ (this remark appears
in \cite{CDM09} and is based on estimates for the norms of certain special
homogeneous polynomials used in \cite[Proposition 8.5]{BBJP}). This result
seems to be uncomfortable since the canonical sequence $(\mathcal{P}%
_{k})_{k=1}^{\infty}$ should be a prototype of the essence of coherence. Using
our forthcoming approach to the notion of coherence, the pair $(\mathcal{P}%
_{k},\mathcal{L}_{k})_{k=1}^{\infty}$ (composed by the ideals of continuous
$n$-homogeneous polynomials and continuous $n$-linear operators, with the
$\sup$ norm) will be coherent and compatible with the ideal of continuous
linear operators.

It is worth mentioning that all these concepts have some connection with the
concept of Property (B) defined in \cite{BBJP} and also with the idea of
holomorphy types, due to L. Nachbin (and with a similar notion of global
holomorphy types, which appears in \cite{BBJP}). In \cite{CDM09} the reader
can find some useful comparisons between the concepts above. As mentioned in
\cite{CDM09}, coherent sequences are always global holomorphy types. It is
interesting to note that in \cite[Example 1.15]{CDM09} the exact example that
fails to be coherent/compatible with the ideal of absolutely summing operators
also fails to be a global holomorphy type. For recent striking applications of
the theory of holomorphy types we refer to \cite{fou}.


\section{Coherence and compatibility: a new approach \label{yu}}

From now on $(\mathcal{U}_{k},\mathcal{M}_{k})_{k=1}^{N}$ is a sequence, where
each $\mathcal{U}_{k}$ is a (quasi-) normed ideal of $k$-homogeneous
polynomials and each $\mathcal{M}_{k}$ is a (quasi-) normed ideal of
$k$-linear mappings. The parameter $N$ can be eventually infinity. Motivated
by the argument that the notions of compatibility and coherence should be
defined simultaneously for polynomials and multilinear mappings, we propose
the following approach:

\begin{definition}
[Compatible pair of ideals]\label{d1}Let $\mathcal{U}$ be a normed operator
ideal and $N\in\left(  \mathbb{N\smallsetminus}\left\{  1\right\}  \right)
\cup\left\{  \infty\right\}  $. A sequence $\left(  \mathcal{U}_{n}%
,\mathcal{M}_{n}\right)  _{n=1}^{N}$ with $\mathcal{U}_{1}=\mathcal{M}%
_{1}=\mathcal{U}$ is compatible with $\mathcal{U}$ if there exist positive
constants $\alpha_{1},\alpha_{2},\alpha_{3},\alpha_{4}$ such that for all
Banach spaces $E$ and $F,$ the following conditions hold for all $n\in\left\{
2,...,N\right\}  $:

(CP 1) If $k\in\left\{  1,\ldots,n\right\}  ,$ $T\in\mathcal{M}_{n}\left(
E_{1},...,E_{n};F\right)  $ and $a_{j}\in E_{j}$ for all $j\in\left\{
1,...,n\right\}  \mathbb{\smallsetminus}\left\{  k\right\}  $, then
\[
T_{a_{1},\ldots,a_{k-1},a_{k+1},\ldots,a_{n}}\in\mathcal{U}\left(
E_{k};F\right)
\]
and
\[
\left\Vert T_{a_{1},\ldots,a_{k-1},a_{k+1},\ldots,a_{n}}\right\Vert
_{\mathcal{U}}\leq\alpha_{1}\left\Vert T\right\Vert _{\mathcal{M}_{n}%
}\left\Vert a_{1}\right\Vert \ldots\left\Vert a_{k-1}\right\Vert \left\Vert
a_{k+1}\right\Vert \ldots\left\Vert a_{n}\right\Vert .
\]
(CP 2) If $P\in\mathcal{U}_{n}\left(  ^{n}E;F\right)  $ and $a\in E$, then
$P_{a^{n-1}}\in\mathcal{U}\left(  E;F\right)  $ and%
\[
\left\Vert P_{a^{n-1}}\right\Vert _{\mathcal{U}}\leq\alpha_{2}\left\Vert
\overset{\vee}{P}\right\Vert _{\mathcal{M}_{n}}\left\Vert a\right\Vert
^{n-1}.
\]
(CP 3) If $u\in\mathcal{U}\left(  E_{n};F\right)  ,$ $\gamma_{j}\in
E_{j}^{\ast}$ for all $j=1,....,n-1$, then
\[
\gamma_{1}\cdots\gamma_{n-1}u\in\mathcal{M}_{n}\left(  E_{1},..,E_{n}%
;F\right)
\]
and
\[
\left\Vert \gamma_{1}\cdots\gamma_{n-1}u\right\Vert _{\mathcal{M}_{n}}%
\leq\alpha_{3}\left\Vert \gamma_{1}\right\Vert ...\left\Vert \gamma
_{n-1}\right\Vert \left\Vert u\right\Vert _{\mathcal{U}}.
\]
(CP 4) If $u\in\mathcal{U}\left(  E;F\right)  ,$ $\gamma\in E^{\ast}$, then
$\gamma^{n-1}u\in\mathcal{U}_{n}\left(  ^{n}E;F\right)  $ and
\[
\left\Vert \gamma^{n-1}u\right\Vert _{\mathcal{U}_{n}}\leq\alpha_{4}\left\Vert
\gamma\right\Vert ^{n-1}\left\Vert u\right\Vert _{\mathcal{U}}.
\]
(CP 5) $P$ belongs to $\mathcal{U}_{n}\left(  ^{n}E;F\right)  $ if, and only
if, $\overset{\vee}{P}$ belongs to $\mathcal{M}_{n}\left(  ^{n}E;F\right)  $.
\end{definition}

\begin{definition}
[Coherent pair of ideals]\label{d2}Let $\mathcal{U}$ be a normed operator
ideal and $N\in\mathbb{N}\cup\left\{  \infty\right\}  .$ A sequence $\left(
\mathcal{U}_{k},\mathcal{M}_{k}\right)  _{k=1}^{N},$ with $\mathcal{U}%
_{1}=\mathcal{M}_{1}=\mathcal{U}$, is coherent if there exist positive
constants $\beta_{1},\beta_{2},\beta_{3},\beta_{4}$ such that for all Banach
spaces $E$ and $F$ the following conditions hold for $k=1,...,N-1$:

(CH 1) If $T\in\mathcal{M}_{k+1}\left(  E_{1},...,E_{k+1};F\right)  $ and
$a_{j}\in E_{j}$ for $j=1,\ldots,k+1,$ then
\[
T_{a_{j}}\in\mathcal{M}_{k}\left(  E_{1},\ldots,E_{j-1},E_{j+1},\ldots
,E_{k+1};F\right)
\]
and
\[
\left\Vert T_{a_{j}}\right\Vert _{\mathcal{M}_{k}}\leq\beta_{1}\left\Vert
T\right\Vert _{\mathcal{M}_{k+1}}\left\Vert a_{j}\right\Vert .
\]
(CH 2) If $P\in\mathcal{U}_{k+1}\left(  ^{k+1}E;F\right)  ,$ $a\in E$, then
$P_{a}$ belongs to $\mathcal{U}_{k}\left(  ^{k}E;F\right)  $ and
\[
\left\Vert P_{a}\right\Vert _{\mathcal{U}_{k}}\leq\beta_{2}\left\Vert
\overset{\vee}{P}\right\Vert _{\mathcal{M}_{k+1}}\left\Vert a\right\Vert .
\]
(CH 3) If $T\in\mathcal{M}_{k}\left(  E_{1},...,E_{k};F\right)  ,$ $\gamma\in
E_{k+1}^{\ast}$, then
\[
\gamma T\in\mathcal{M}_{k+1}\left(  E_{1},...,E_{k+1};F\right)  \text{ and
}\left\Vert \gamma T\right\Vert _{\mathcal{M}_{k+1}}\leq\beta_{3}\left\Vert
\gamma\right\Vert \left\Vert T\right\Vert _{\mathcal{M}_{k}}.
\]
(CH 4)If $P\in\mathcal{U}_{k}\left(  ^{k}E;F\right)  ,$ $\gamma\in E^{\ast}$,
then
\[
\gamma P\in\mathcal{U}_{k+1}\left(  ^{k+1}E;F\right)  \text{ and }\left\Vert
\gamma P\right\Vert _{\mathcal{U}_{k+1}}\leq\beta_{4}\left\Vert \gamma
\right\Vert \left\Vert P\right\Vert _{\mathcal{U}_{k}}.
\]
(CH 5) For all $k=1,...,N,$ $P$ belongs to $\mathcal{U}_{k}\left(
^{k}E;F\right)  $ if, and only if, $\overset{\vee}{P}$ belongs to
$\mathcal{M}_{k}\left(  ^{k}E;F\right)  $.
\end{definition}

\begin{remark}
\label{yv}Note that Definition \ref{d1} is quite different from the concept of
compatibility from Carando\textit{ et al}. For example, our approach asks for
universal constants $\alpha_{1},\alpha_{2},\alpha_{3},\alpha_{4}$ (that not
depend on $n$). It is also worth mentioning that a coherent sequence $\left(
\mathcal{U}_{k},\mathcal{M}_{k}\right)  _{k=1}^{N}$ is not necessarily
compatible with $\mathcal{U}_{1}$. If $\beta_{1}=\beta_{2}=\beta_{3}=\beta
_{4}=1$ then the coherence of a sequence $\left(  \mathcal{U}_{k}%
,\mathcal{M}_{k}\right)  _{k=1}^{N}$ easily implies in the compatibility with
$\mathcal{U}_{1}$. In the Example \ref{ffiiu} we also show that a sequence
$\left(  \mathcal{U}_{k},\mathcal{M}_{k}\right)  _{k=1}^{N}$ which is
compatible with $\mathcal{U}_{1}$ may fail to be coherent.

A weaker concept of coherence of pairs, where the constants $\alpha_{1}%
,\alpha_{2},\alpha_{3},\alpha_{4}$ may depend on $k$ may be also interesting.
\end{remark}

As we have mentioned before, for the real case the canonical sequence
$(\mathcal{P}_{k})_{k=1}^{\infty}$ is not coherent according to the original
definition of Carando, Dimant and Muro. The following straightforward
proposition shows that the situation is different according to the new approach:

\begin{proposition}
The pair $(\mathcal{P}_{k},\mathcal{L}_{k})_{k=1}^{\infty}$ (composed by the
ideals of continuous $n$-homogeneous polynomials and continuous $n$-linear
operators, with the $\sup$ norm) is coherent and compatible with the ideal of
continuous linear operators.
\end{proposition}

\section{The Factorisation Method\label{jt0}}

The factorisation method is an important abstract way of extending an operator
ideal to polynomials and multilinear mappings. In this section we show that
the sequence of pairs obtained by this method is coherent and compatible with
the original ideal.

For an operator ideal $\mathcal{I}$, an $n$-linear mapping $A\in
\mathcal{L}\left(  E_{1},\ldots,E_{n};F\right)  $ is of type $\mathcal{L}%
\left(  ^{n}\mathcal{I}\right)  $ if there are Banach spaces $G_{1}%
,\ldots,G_{n}$, linear operators $u_{j}\in\mathcal{I}\left(  E_{j}%
,G_{j}\right)  $, $j=1,\ldots,n$ and $B\in\mathcal{L}\left(  G_{1}%
,\ldots,G_{n};F\right)  $ such that
\begin{equation}
A=B\circ\left(  u_{1},\ldots,u_{n}\right)  . \label{cc}%
\end{equation}
In this case we write
\[
A\in\mathcal{L}\left(  ^{n}\mathcal{I}\right)  \left(  E_{1},\ldots
,E_{n};F\right)  ,
\]
and define
\[
\left\Vert A\right\Vert _{\mathcal{L}\left(  ^{n}\mathcal{I}\right)  }%
=\inf\left\Vert B\right\Vert \left\Vert u_{1}\right\Vert _{\mathcal{I}}%
,\ldots,\left\Vert u_{n}\right\Vert _{\mathcal{I}},
\]
where the infimum is taken over all possible factorisations (\ref{cc}).

The following definition will be useful:

\begin{definition}
\label{or}If $\mathcal{M}=\left(  \mathcal{M}_{n}\right)  _{n=1}^{\infty}$ is
a (quasi-) normed multi-ideal, $(\mathcal{P}_{\mathcal{M}}^{n})_{n=1}^{\infty
}$ is a sequence such that $P\in\mathcal{P}_{\mathcal{M}}^{n}(^{n}E;F)$ if and
only if $\overset{\vee}{P}\in\mathcal{M}_{n}(^{n}E;F).$ Moreover%
\[
\left\Vert P\right\Vert _{\mathcal{P}_{\mathcal{M}}^{n}}:=\left\Vert
\overset{\vee}{P}\right\Vert _{\mathcal{M}_{n}}.
\]

\end{definition}

\begin{proposition}
(see \cite[page 46]{BBJP}) If $\mathcal{M}=\left(  \mathcal{M}_{n}\right)
_{n=1}^{\infty}$ is a (quasi-) Banach multi-ideal then $\mathcal{P}%
_{\mathcal{M}}:=(\mathcal{P}_{\mathcal{M}}^{k})_{k=1}^{\infty}$ is a (quasi-)
Banach polynomial ideal.
\end{proposition}

For all $n$, $\mathcal{L}\left(  ^{n}\mathcal{I}\right)  $ is a complete
$\left(  1/n\right)  $-normed ideal of $n$-linear mappings; hence $\left(
\mathcal{L}\left(  ^{n}\mathcal{I}\right)  \right)  _{n=1}^{\infty}$ is a
quasi-Banach multi-ideal and $\left(  \mathcal{P}_{\mathcal{L}\left(
^{n}\mathcal{I}\right)  }^{n}\right)  _{n=1}^{\infty},$ constructed as in
Definition \ref{or}, is a quasi-Banach polynomial ideal.

The proof that the sequence $\left(  \left(  \mathcal{P}_{\mathcal{L}\left(
^{n}\mathcal{I}\right)  }^{n},\left\Vert .\right\Vert _{\mathcal{P}%
_{\mathcal{L}\left(  ^{n}\mathcal{I}\right)  }^{n}}\right)  ,\left(
\mathcal{L}\left(  ^{n}\mathcal{I}\right)  ,\left\Vert .\right\Vert
_{\mathcal{L}\left(  ^{n}\mathcal{I}\right)  }\right)  \right)  _{n=1}%
^{\infty}$ is coherent and compatible with the ideal $\mathcal{I}$\ will be an
immediate consequence of the next results.

\begin{proposition}
\label{FatoracaoCoerente1} If $P\in\mathcal{P}_{\mathcal{L}\left(
^{n+1}\mathcal{I}\right)  }^{n+1}\left(  ^{n+1}E;F\right)  $ and $a\in E$,
then $P_{a}$ belongs to $\mathcal{P}_{\mathcal{L}\left(  ^{n}\mathcal{I}%
\right)  }^{n}\left(  ^{n}E;F\right)  $ and
\[
\left\Vert P_{a}\right\Vert _{\mathcal{P}_{\mathcal{L}\left(  ^{n}%
\mathcal{I}\right)  }^{n}}\leq\left\Vert \overset{\vee}{P}\right\Vert
_{\mathcal{L}\left(  ^{n+1}\mathcal{I}\right)  }\left\Vert a\right\Vert .
\]

\end{proposition}

\begin{proof}
There exist Banach spaces $G_{1},...,G_{n+1}$, linear operators $u_{j}%
\in\mathcal{I}_{j}\left(  E;G_{j}\right)  $, $j=1,\ldots,n+1$, and a
multilinear mapping $B\in\mathcal{L}\left(  G_{1},...,G_{n+1};F\right)  $ such
that
\begin{equation}
\overset{\vee}{P}\left(  x_{1},\ldots,x_{n+1}\right)  =B\left(  u_{1}\left(
x_{1}\right)  ,\ldots,u_{n+1}\left(  x_{n+1}\right)  \right)  . \label{pqwe}%
\end{equation}
So
\begin{align*}
\left(  P_{a}\right)  ^{\vee}\left(  x_{1},...,x_{n}\right)   &
=\overset{\vee}{P}\left(  x_{1},\ldots,x_{n},a\right)  =B\left(  u_{1}\left(
x_{1}\right)  ,\ldots,u_{n}(x_{n}),u_{n+1}\left(  a\right)  \right) \\
&  =B_{u_{n+1}\left(  a\right)  }\left(  u_{1}\left(  x_{1}\right)
,\ldots,u_{n}\left(  x_{n}\right)  \right)  .
\end{align*}
Hence, since $\left\Vert u_{n+1}\right\Vert \leq\left\Vert u_{n+1}\right\Vert
_{\mathcal{I}},$ we have
\[
\left\Vert P_{a}\right\Vert _{\mathcal{P}_{\mathcal{L}\left(  ^{n}%
\mathcal{I}\right)  }^{n}}=\left\Vert \left(  P_{a}\right)  ^{\vee}\right\Vert
_{\mathcal{L}\left(  ^{n}\mathcal{I}\right)  }\leq\left\Vert B_{u_{n+1}\left(
a\right)  }\right\Vert
{\textstyle\prod\limits_{j=1}^{n}}
\left\Vert u_{j}\right\Vert _{\mathcal{I}}\leq\left\Vert a\right\Vert
\left\Vert B\right\Vert
{\textstyle\prod\limits_{j=1}^{n+1}}
\left\Vert u_{j}\right\Vert _{\mathcal{I}}.
\]
for all representation (\ref{pqwe}) and the proof is completed when
considering the infimum over all such representations.
\end{proof}

The next result is inspired in \cite[Proposition 3.1]{CDM09}:

\begin{proposition}
\label{FatoracaoCoerente2} If $P\in\mathcal{P}_{\mathcal{L}\left(
^{n}\mathcal{I}\right)  }^{n}\left(  ^{n}E;F\right)  $ and $\gamma\in E^{\ast
}$, then $\gamma P\in$ $\mathcal{P}_{\mathcal{L}\left(  ^{n+1}\mathcal{I}%
\right)  }\left(  ^{n+1}E;F\right)  $ and
\[
\left\Vert \gamma P\right\Vert _{\mathcal{P}_{\mathcal{L}\left(
^{n+1}\mathcal{I}\right)  }^{n+1}}\leq\left\Vert \gamma\right\Vert \left\Vert
P\right\Vert _{\mathcal{L}\left(  ^{n}\mathcal{I}\right)  }.
\]

\end{proposition}

\begin{proof}
We can suppose $\left\Vert \gamma\right\Vert =1$. There exist Banach spaces
$G_{1},\ldots,G_{n}$, a multilinear mapping $B\in\mathcal{L}\left(
G_{1},\ldots,G_{n};F\right)  $, and linear operators $u_{j}\in$ $\mathcal{I}%
\left(  E;G_{j}\right)  $, $j=1,\ldots,n$\ such that
\[
\overset{\vee}{P}=B\circ\left(  u_{1},\ldots,u_{n}\right)  .
\]
Now, consider the mapping $\tilde{B}\in\mathcal{L}\left(  G_{1},\ldots
,G_{n},\mathbb{K};F\right)  $ defined by
\[
\tilde{B}\left(  y_{1},\ldots,y_{n},\gamma\left(  x\right)  \right)
=\gamma\left(  x\right)  B\left(  y_{1},\ldots,y_{n}\right)  .
\]
Observe that $\tilde{B}$ is well-defined, and that
\begin{align*}
\tilde{B}\left(  u_{1}\left(  x_{1}\right)  ,\ldots,u_{n}\left(  x_{n}\right)
,\gamma\left(  x_{n+1}\right)  \right)   &  =\gamma\left(  x_{n+1}\right)
B\left(  u_{1}\left(  x_{1}\right)  ,\ldots,u\left(  x_{n}\right)  \right) \\
&  =\gamma\left(  x_{n+1}\right)  \overset{\vee}{P}\left(  x_{1},\ldots
,x_{n}\right)  .
\end{align*}
Therefore,
\[
\gamma\overset{\vee}{P}\in\mathcal{L}\left(  ^{n+1}\mathcal{I}\right)  \left(
^{n+1}E;F\right)  .
\]
Since $\left\Vert \tilde{B}\right\Vert =\left\Vert B\right\Vert $ and%
\[
\left(  \gamma P\right)  ^{\vee}(x_{1},...,x_{n+1})=\frac{\gamma
(x_{1})\overset{\vee}{P}(x_{2},...,x_{n+1})+...+\gamma(x_{n+1})\overset{\vee
}{P}(x_{1},...,x_{n})}{n+1},
\]
we have
\[
\left\Vert \gamma P\right\Vert _{\mathcal{P}_{\mathcal{L}\left(
^{n+1}\mathcal{I}\right)  }^{n+1}}=\left\Vert \left(  \gamma P\right)  ^{\vee
}\right\Vert _{\mathcal{L}\left(  ^{n+1}\mathcal{I}\right)  }\leq\frac{1}%
{n+1}\left(  \left(  n+1\right)  \left\Vert \gamma\overset{\vee}{P}\right\Vert
_{\mathcal{L}\left(  ^{n+1}\mathcal{I}\right)  }\right)  \leq\left\Vert
B\right\Vert
{\textstyle\prod\limits_{j=1}^{n}}
\left\Vert u_{j}\right\Vert _{\mathcal{I}}%
\]
and the proof is completed$.$
\end{proof}

\begin{proposition}
\label{FatoracaoCoerente3} Let $k\in\left\{  1,\ldots,n+1\right\}  $. If
$T\in\mathcal{L}\left(  ^{n+1}\mathcal{I}\right)  \left(  E_{1},...,E_{n+1}%
;F\right)  $ and $a_{k}\in E_{k}$, then
\[
T_{a_{k}}\in\mathcal{L}\left(  ^{n}\mathcal{I}\right)  \left(  E_{1}%
,\ldots,E_{k-1},E_{k+1},\ldots,E_{n+1};F\right)
\]
and
\[
\left\Vert T_{a_{k}}\right\Vert _{\mathcal{L}\left(  ^{n}\mathcal{I}\right)
}\leq\left\Vert T\right\Vert _{\mathcal{L}\left(  ^{n+1}\mathcal{I}\right)
}\left\Vert a_{k}\right\Vert .
\]

\end{proposition}

\begin{proof}
There are Banach spaces $G_{1},\ldots,G_{n+1}$, linear operators $u_{j}%
\in\mathcal{I}\left(  E_{j};G_{j}\right)  $, $j=1,\ldots,n+1$, and
$B\in\mathcal{L}\left(  G_{1},\ldots,G_{n+1};F\right)  $ such that
\[
T\left(  x_{1},\ldots,x_{n+1}\right)  =B\left(  u_{1}\left(  x_{1}\right)
,\ldots,u_{n+1}\left(  x_{n+1}\right)  \right)  .
\]
Hence,
\begin{align*}
&  T_{a_{k}}\left(  x_{1},\ldots,x_{k-1},x_{k+1},\ldots x_{n+1}\right) \\
&  =B_{u_{k}\left(  a_{k}\right)  }\left(  u_{1}\left(  x_{1}\right)
,\ldots,u_{k-1}\left(  x_{k-1}\right)  ,u_{k+1}\left(  x_{k+1}\right)
,\ldots,u_{n+1}\left(  x_{n+1}\right)  \right)
\end{align*}
and the proof follows the lines of the proof of Proposition
\ref{FatoracaoCoerente2}.
\end{proof}

The proof of the next proposition is essentially the same of Proposition
\ref{FatoracaoCoerente2}, and we omit:

\begin{proposition}
\label{FatoracaoCoerente4} If $T\in\mathcal{L}\left(  ^{n}\mathcal{I}\right)
\left(  E_{1},\ldots,E_{n};F\right)  $ and $\gamma\in E_{n+1}^{\ast}$, then
\[
\gamma T\in\mathcal{L}\left(  ^{n+1}\mathcal{I}\right)  \left(  E_{1}%
,...,E_{n+1};F\right)  \text{ and }\left\Vert \gamma T\right\Vert
_{\mathcal{L}\left(  ^{n+1}\mathcal{I}\right)  }\leq\left\Vert \gamma
\right\Vert \left\Vert T\right\Vert _{\mathcal{L}\left(  ^{n}\mathcal{I}%
\right)  }.
\]

\end{proposition}

From the previous results and Remark \ref{yv} we have:

\begin{theorem}
The sequence $\left(  \left(  \mathcal{P}_{\mathcal{L}\left(  ^{n}%
\mathcal{I}\right)  }^{n},\left\Vert .\right\Vert _{\mathcal{P}_{\mathcal{L}%
\left(  ^{n}\mathcal{I}\right)  }^{n}}\right)  ,\left(  \mathcal{L}\left(
^{n}\mathcal{I}\right)  ,\left\Vert .\right\Vert _{\mathcal{L}\left(
^{n}\mathcal{I}\right)  }\right)  \right)  _{n=1}^{\infty}$ is coherent and
compatible with the ideal $\mathcal{I}$.
\end{theorem}

\section{Nonlinear variants of absolutely summing operators: a brief summary}

The class of absolutely $p$-summing linear operators is one of the most
successful examples of operator ideals, having its special space in several
textbooks related to Banach Space Theory (we refer the reader to \cite{AK,
Diestel, Fabian, LT, Ryan, Woy}). Its roots go back to the 1950s, when
Grothendieck still worked in Functional Analysis (see the famous
R\'{e}sum\'{e} \cite{Gro1953} and also \cite{grote, grothe3, grothe2, Gro1955}
for other papers of Grothendieck in Functional Analysis). The classical papers
of Lindenstrauss--Pe\l czy\'{n}ski \cite{LP} and Pietsch \cite{stu} were also
fundamental for the development of the theory.

As we have mentioned before, there are various possible multilinear and
polynomial approaches to the notion of absolutely summing operators; for this
reason we think that this is a nice context to test the notions of coherence
and compatibility of pairs previously defined.

For $1\leq p<\infty$, we denote by $\ell_{p}^{w}(E)$ the space composed by the
sequences $\left(  x_{j}\right)  _{j=1}^{\infty}$ in $E$ so that $\left(
\varphi\left(  x_{j}\right)  \right)  _{j=1}^{\infty}\in\ell_{p}$ for all
continuous linear functionals $\varphi:E\rightarrow\mathbb{K}$. The space
$\ell_{p}^{w}(E)$ is a Banach space when endowed with the norm $\Vert
\cdot\Vert_{w,p}$ given by
\[
\left\Vert \left(  x_{j}\right)  _{j=1}^{\infty}\right\Vert _{w,p}%
:=\sup_{\varphi\in B_{E^{\ast}}}\left\Vert \left(  \varphi\left(
x_{j}\right)  \right)  _{j=1}^{\infty}\right\Vert _{p}.
\]
The subspace of $\ell_{p}^{w}(E)$ of all sequences $\left(  x_{j}\right)
_{j=1}^{\infty}\in\ell_{p}^{w}\left(  E\right)  $ such that $\lim
_{m\rightarrow\infty}\left\Vert \left(  x_{j}\right)  _{j=m}^{\infty
}\right\Vert _{w,p}=0$ is denoted by $\ell_{p}^{u}\left(  E\right)  $. If
$1\leq q\leq p<\infty$ a continuous linear operator $u:E\rightarrow F$ is
absolutely $(p,q)$-summing if there is a constant $C\geq0$ such that%
\[
\left(
{\displaystyle\sum\limits_{j=1}^{\infty}}
\left\Vert u(x_{j})\right\Vert ^{p}\right)  ^{1/p}\leq C\left\Vert
(x_{j})_{j=1}^{\infty}\right\Vert _{w,q}%
\]
for all $(x_{j})_{j=1}^{\infty}\in\ell_{q}^{u}(E)$. We denote the set of all
absolutely $(p,q)$-summing operators by $\Pi_{p,q}$ ($\Pi_{p}$ if $p=q$) and
the space of all absolutely $(p,q)$-summing operators from $E$ to $F$ by
$\Pi_{p,q}\left(  E,F\right)  .$ The infimum of all $C$ that satisfy the
inequality above defines a norm on $\Pi_{p,q}\left(  E,F\right)  $, denoted by
$\left\Vert .\right\Vert _{as(p,q)}$ (or $\left\Vert .\right\Vert _{as,p}$ if
$p=q$) and $\left(  \Pi_{p,q},\left\Vert .\right\Vert _{as(p,q)}\right)  $ is
a Banach operator ideal. The notion of absolutely summing operators is due to
Pietsch \cite{stu}. For a complete panorama on the theory of absolutely
summing operators we refer the reader to the classical monograph
\cite{Diestel} and classical papers \cite{ben2, Dies, DPR,LP, stu}; for recent
developments we refer the reader to \cite{PellZ, kit, ku, ku2, pu} and
references therein.

The adequate extension of the linear theory of absolutely summing operators to
the multilinear setting is a complicated matter; there are different
approaches and different lines of investigation.

Historically, in some sense, the multilinear theory of absolutely summing
mappings seems to have its starting point in 1930, with Littlewood's $4/3$
Theorem \cite{LLL} and, one year later, with the Bohnenblust--Hille Theorem
\cite{BH}. The Bohnenblust--Hille Theorem was overlooked for a long time and
only in the the 80's the interest in the multilinear theory related to
absolutely summing operators was recovered, motivated by A. Pietsch's work
\cite{PPPP}. In the present paper we deal with some of the most usual
polynomial and multilinear extensions of $\Pi_{p}$ (strongly summing
multilinear operators, multiple summing multilinear operators, absolutely
summing multilinear operators, dominated multilinear operators, everywhere
absolutely summing multilinear operators and their polynomial versions). For
more details concerning the nonlinear theory of absolutely summing operators
and recent developments and applications we refer the reader to \cite{BBJP,
BPR1, cg, Def2, dimant, Matos-N, Studia, phi} and references therein.

A polynomial $P\in\mathcal{P}_{n}(^{n}E;F)$ is $(p;q)$-summing at $a\in E$ if
$\left(  P(a+x_{j})-P(a)\right)  _{j=1}^{\infty}\in\ell_{p}(F)$ for all
$(x_{j})_{j=1}^{\infty}\in\ell_{q}^{u}(E)$. If $1\leq q\leq p<\infty,$ the
space composed by the $n$-homogeneous polynomials from $E$ to $F$ that are
$(p;q)$-summing at every point is denoted by $\mathcal{P}_{as(p;q)}%
^{n,ev}(^{n}E;F).$ The polynomials in $\mathcal{P}_{as(p;q)}^{n,ev}$ are
called everywhere absolutely summing.

M.C. Matos \cite{Matos-N} defined a norm on the space $\mathcal{P}%
_{as(p;q)}^{n,ev}(^{n}E;F)$ of everywhere $(p;q)$-summing polynomials by
considering the polynomial $\Psi_{p;q}(P)\colon\ell_{q}^{u}(E)\longrightarrow
\ell_{p}(F)$ given by
\[
(x_{j})_{j=1}^{\infty}\longmapsto(P(x_{1}),(P(x_{1}+x_{j})-P(x_{1}%
))_{j=2}^{\infty})
\]
and showing that the correspondence $P\longrightarrow\Vert\Psi_{p;q}(P)\Vert$
defines a norm on $\mathcal{P}_{as(p;q)}^{n,ev}(^{n}E;F)$. From now on this
norm is denoted by $\left\Vert \cdot\right\Vert _{ev^{(1)}(p;q)}$. Matos also
proved that this norm is complete and that $(\mathcal{P}_{as(p;q)}%
^{n,ev},\left\Vert \cdot\right\Vert _{ev^{(1)}(p;q)})$ is a global holomorphy
type. From \cite{Scand} it is known that
\[
\lim_{n\rightarrow\infty}\Vert P_{n}:\mathbb{K}\rightarrow\mathbb{K};\text{
}P_{n}\left(  \lambda\right)  =\lambda^{n}\Vert_{ev^{(1)}(p;q)}=\infty
\]
and this estimate will allow us to conclude that $(\mathcal{P}_{as(p;q)}%
^{n,ev},\left\Vert \cdot\right\Vert _{ev^{(1)}(p;q)})_{n=1}^{\infty}$ is
\textquotedblleft compatible, in the sense of Carando \textit{et
al.}\textquotedblright\ with no operator ideal; here we have used the term
\textquotedblleft compatible , in the sense of Carando \textit{et
al.}\textquotedblright\ in a more general form, since the sequence
$(\mathcal{P}_{as(p;q)}^{n,ev},\left\Vert \cdot\right\Vert _{ev^{(1)}%
(p;q)})_{n=1}^{\infty}$ is not exactly a normed polynomial ideal (since it
fails (P2)), but just a global holomorphy type.

\begin{proposition}
If $p\geq q>1$ and $n\geq2$, then $(\mathcal{P}_{as(p;q)}^{n,ev},\left\Vert
\cdot\right\Vert _{ev^{(1)}(p;q)})$ is \textquotedblleft
compatible\textquotedblright\ with no operator ideal.
\end{proposition}

\begin{proof}
Given $n\in\mathbb{N}$, let $P_{n}\colon\mathbb{K}\longrightarrow\mathbb{K}$
denote the trivial $n$-homogeneous polynomial given by $P_{n}(\lambda
)=\lambda^{n}$. From \cite[Proposition 4.4]{Scand} it is also known that
\[
\lim_{n\rightarrow\infty}\Vert P_{n}\Vert_{ev^{(1)}(p;q)}=\infty
\]
for all $p\geq q>1$. So, by considering $u=\gamma=P_{1},$ we conclude that
$\gamma^{n-1}u$ belongs to $\mathcal{P}_{as(p;q)}^{n,ev}(^{n}E;\mathbb{K})$
and
\[
\lim_{n\rightarrow\infty}\left\Vert \gamma^{n-1}u\right\Vert _{ev^{(1)}%
(p;q)}=\lim_{n\rightarrow\infty}\Vert P_{n}\Vert_{ev^{(1)}(p;q)}=\infty
\]
and, on the other hand, for every operator ideal $\mathcal{I},$ we have%
\[
\left\Vert \gamma\right\Vert ^{n-1}\left\Vert u\right\Vert _{\mathcal{I}%
}<\infty.
\]
We thus conclude that $(\mathcal{P}_{as(p;q)}^{n,ev},\left\Vert \cdot
\right\Vert _{ev^{(1)}(p;q)})$ is compatible with no operator ideal.
\end{proof}

In the result above the non-compatibility is a fault of the norm and by
considering the similar concept for multilinear mappings, the respective pair
of ideals will also fail to be compatible with respect to our new approach.
The situation will be different when considering a new norm in the next section.

From now on we will adopt the classical notation for the spaces of continuous
$n$-homogeneous polynomials and continuous $n$-linear mappings. More
precisely, we will write $\mathcal{P}\left(  ^{n}E;F\right)  $ and
$\mathcal{L}\left(  E_{1},..,E_{n};F\right)  $ instead of $\mathcal{P}%
_{n}\left(  ^{n}E;F\right)  $ and $\mathcal{L}_{n}\left(  E_{1},..,E_{n}%
;F\right)  ,$ respectively. When $E_{1}=\cdots=E_{n}$ we will write
$\mathcal{L}\left(  ^{n}E;F\right)  .$

\section{Everywhere absolutely summing multilinear operators and polynomials}

One of the first multilinear generalisations of the ideal of absolutely
summing operators (see \cite{AM}) is the following: If $0<p,q<\infty$ and
$p\geq nq$ an $n$-homogeneous polynomial $P\in\mathcal{P}(^{n}E;F)$ is
absolutely $(p;q)$-summing if there exists $C$ such that%
\begin{equation}
\left(  \sum_{j=1}^{\infty}\left\Vert P(x_{j})\right\Vert ^{p}\right)
^{1/p}\leq C\left\Vert (x_{j})_{j=1}^{\infty})\right\Vert _{w,q}^{n}
\label{pmp}%
\end{equation}
for all $(x_{j})_{j=1}^{\infty}\in\ell_{q}^{u}(E)$. The infimum of all $C$ for
which (\ref{pmp}) holds defines a complete norm (if $p\geq1$), defined by
$\left\Vert .\right\Vert _{as(p;q)},$ on $\mathcal{P}_{as(p;q)}^{n}(^{n}E;F)$.
If $p=q$ we write $\mathcal{P}_{as,p}^{n}$ instead of $\mathcal{P}%
_{as(p;q)}^{n}.$

It is not difficult to show that the definition above is equivalent to saying
that $\left(  P(x_{j})\right)  _{j=1}^{\infty}\in\ell_{p}(F)$ for all
$(x_{j})_{j=1}^{\infty}\in\ell_{q}^{u}(E)$. This ideal, however, is not closed
under differentiation and is not a (global) holomorphy type. Besides, in this
case the spirit of the linear ideal is also destroyed by several coincidence
theorems, which have no relation with the linear case. For example, using that
$\ell_{p}$ (for $1\leq p\leq2$) has cotype $2$ it is easy to show that
\begin{equation}
\mathcal{P}_{as,1}^{n}(^{n}\ell_{p};F)=\mathcal{P}(^{n}\ell_{p};F) \label{kio}%
\end{equation}
for all $p\in\lbrack1,2],$ $n\geq2$ and all $F$; these results are far from
being true for $n=1$. So it should be expected that $\left(  \mathcal{P}%
_{as,1}^{n},\left\Vert .\right\Vert _{as,1}\right)  $ is not classified as
compatible with the ideal $\Pi_{1}.$ Similar defects can be found in this
ideal for the general case of $\mathcal{P}_{as(p,q)}^{n}.$ A similar concept
of absolutely summability exists for $n$-linear operators:

An $n$-linear operator $T\in\mathcal{L}(E_{1},\ldots,E_{n};F)$ is absolutely
$(p;q)$-summing (with $p\geq nq$) if there exists $C\geq0$ such that
\begin{equation}
\left(  \sum_{j=1}^{\infty}\left\Vert T(x_{j}^{(1)},\ldots,x_{j}%
^{(n)})\right\Vert ^{p}\right)  ^{1/p}\leq C%
{\displaystyle\prod\limits_{k=1}^{n}}
\left\Vert (x_{j}^{(k)})_{j=1}^{\infty}\right\Vert _{w,q}\nonumber
\end{equation}
for all $(x_{j}^{(k)})_{j=1}^{\infty}\in\ell_{q}^{u}(E_{k}),\,k=1,\ldots,n$.
Moreover, the infimum of all $C$ for which the inequality holds defines a
complete norm (if $p\geq1$)$,$ denoted by $\left\Vert .\right\Vert
_{as(p;q)},$ for this class. This definition is equivalent to saying that
$\left(  T(x_{j}^{(1)},\ldots,x_{j}^{(n)})\right)  _{j=1}^{\infty}$ belongs to
$\ell_{p}(F)$ for all $(x_{j}^{(k)})_{j=1}^{\infty}\in\ell_{q}^{u}(E_{k}).$

This class, denoted by $\mathcal{L}_{as(p,q)}^{n},$ forms a Banach multi-ideal
but defects similar to those of $\mathcal{P}_{as(p,q)}^{n}$ can be easily
found. So we easily have:

\begin{example}
The sequence $\left(  \left(  \mathcal{P}_{as,1}^{n},\left\Vert .\right\Vert
_{as,1}\right)  ,\left(  \mathcal{L}_{as,1}^{n},\left\Vert .\right\Vert
_{as,1}\right)  \right)  _{n=1}^{\infty}$ is not coherent and not compatible
with $\Pi_{1}$.
\end{example}

The main problems of the classes above disappear when we slightly modify their
definitions, as we see in the next definition:

\begin{definition}
Let $1\leq q\leq p<\infty.$ An $n$-linear operator $T\in\mathcal{L}%
(E_{1},\ldots,E_{n};F)$ is everywhere absolutely $(p;q)$-summing (notation
$\mathcal{L}_{as(p;q)}^{n,ev}(E_{1},\ldots,E_{n};F)$) if there exists $C\geq0$
such that
\begin{equation}
\left(  \sum_{j=1}^{\infty}\left\Vert T(a_{1}+x_{j}^{(1)},\ldots,a_{n}%
+x_{j}^{(n)})-T(a_{1},\ldots,a_{n})\right\Vert ^{p}\right)  ^{1/p}\leq C%
{\displaystyle\prod\limits_{k=1}^{n}}
\left(  \left\Vert a_{k}\right\Vert +\left\Vert (x_{j}^{(k)})_{j=1}^{\infty
})\right\Vert _{w,q}\right) \nonumber
\end{equation}
for all $(a_{1},\ldots,a_{n})\in E_{1}\times\cdots\times E_{n}$ and
$(x_{j}^{(k)})_{j=1}^{\infty}\in\ell_{q}^{u}(E_{k}),\,k=1,\ldots,n$.
\ Moreover, the infimum of all $C$ for which the inequality holds defines a
complete norm on $\mathcal{L}_{as(p;q)}^{n,ev}$ denoted by $\left\Vert
\cdot\right\Vert _{ev^{(2)}(p;q)}$.
\end{definition}

The definition above is justified by the following result \cite[Theorem
4.1]{Scand}:

\begin{theorem}
The following assertions are equivalent for $T\in\mathcal{L}(E_{1}%
,\ldots,E_{n};F)$:

\textrm{(a)} $T\in\mathcal{L}_{as(p;q)}^{n,ev}(E_{1},\ldots,E_{n};F).$

\textrm{(b)} The sequence $\left(  T(a_{1}+x_{j}^{(1)},...,a_{n}+x_{j}%
^{(n)})-T(a_{1},...,a_{n})\right)  _{j=1}^{\infty}\in$ $\ell_{p}(F)$ for all
$(x_{j}^{(k)})_{j=1}^{\infty}\in\ell_{p}^{u}(E_{k})$ and all $(a_{1}%
,...,a_{n})\in E_{1}\times\cdots\times E_{n}$.
\end{theorem}

For polynomials the definition and characterisation are similar (and
equivalent, modulo norms), to the definition presented in Section 3:

\begin{definition}
Let $1\leq q\leq p<\infty.$ A polynomial $P\in\mathcal{P}(^{n}E;F)$ is
everywhere absolutely $(p;q)$-summing (notation $\mathcal{P}_{as(p;q)}%
^{n,ev}(^{n}E;F)$) if there exists $C\geq0$ such that%
\[
\left(  \sum_{j=1}^{\infty}\left\Vert P(a+x_{j})-P(a)\right\Vert ^{p}\right)
^{1/p}\leq C\left(  \left\Vert a\right\Vert +\left\Vert (x_{j})_{j=1}^{\infty
})\right\Vert _{w,q}\right)  ^{n}%
\]
for all $a\in E$ and $(x_{j})_{j=1}^{\infty}\in\ell_{q}^{u}(E)$. Moreover, the
infimum of all $C$ for which the inequality holds defines a complete norm on
$\mathcal{P}_{as(p;q)}^{n,ev}(^{n}E;F)$ denoted by $\left\Vert .\right\Vert
_{ev^{(2)}(p;q)}$.
\end{definition}

As in the case of multilinear operators, the following characterisation holds
\cite[Theorem 4.2]{Scand}:

\begin{theorem}
The following assertions are equivalent for $P\in\mathcal{P}(^{n}E;F)$:

\textrm{(a)} $P\in\mathcal{P}_{as(p;q)}^{n,ev}(^{n}E;F).$

\textrm{(b)} The sequence $\left(  P(a+x_{j})-P(a)\right)  _{j=1}^{\infty}%
\in\ell_{p}(F)$ for all $(x_{j})_{j=1}^{\infty}\in\ell_{p}^{u}(E)$ and all
$a\in E$.
\end{theorem}

It was proved in \cite[Proposition 4.3]{Scand} that
\[
\Vert A_{n}\colon\mathbb{K}^{n}\longrightarrow\mathbb{K}:A_{n}(\lambda
_{1},\ldots,\lambda_{n})=\lambda_{1}\cdots\lambda_{n}\Vert_{ev^{(2)}(p;q)}=1
\]
for all $p\geq q\geq1.$ So, it is not difficult to show that $(\mathcal{L}%
_{as(p;q)}^{n,ev},\Vert.\Vert_{ev^{(2)}(p;q)})$ is a Banach multi-ideal:

\begin{proposition}
$(\mathcal{L}_{as(p;q)}^{n,ev},\Vert.\Vert_{ev^{(2)}(p;q)})_{n=1}^{\infty}$ is
a Banach multi-ideal.
\end{proposition}

\begin{proof}
Let $u_{j}\in\mathcal{L}\left(  G_{j},E_{j}\right)  $, $j=1,\ldots,n$,
$T\in\mathcal{L}_{as(p;q)}^{n,ev}\left(  E_{1},\ldots,E_{n};F\right)  $ and
$w\in\mathcal{L}(F;G)$. Note that%
\begin{align*}
&  \left(  {\sum\limits_{j=1}^{\infty}}\left\Vert w\circ T\circ\left(
u_{1},\ldots,u_{n}\right)  \left(  a_{1}+x_{j}^{(1)},\ldots,a_{n}+x_{j}%
^{(n)}\right)  -w\circ T\circ\left(  u_{1},\ldots,u_{n}\right)  \left(
a_{1},\ldots,a_{n}\right)  \right\Vert ^{p}\right)  ^{1/p}\\
&  \leq\left\Vert w\right\Vert \left(  {\sum\limits_{j=1}^{\infty}}\left\Vert
T\left(  u_{1}(a_{1})+u_{1}(x_{j}^{(1)}),\ldots,u_{n}(a_{n})+u_{n}(x_{j}%
^{(n)})\right)  -T\left(  u_{1}(a_{1}),\ldots,u_{n}(a_{n})\right)  \right\Vert
^{p}\right)  ^{1/p}\\
&  \leq\left\Vert w\right\Vert \left\Vert T\right\Vert _{ev^{(2)}(p;q)}%
{\textstyle\prod\limits_{k=1}^{n}}
\left(  \left\Vert u_{k}(a_{k})\right\Vert +\left\Vert \left(  u_{k}\left(
x_{j}^{(k)}\right)  \right)  _{j=1}^{\infty}\right\Vert _{w,q}\right) \\
&  \leq\left\Vert w\right\Vert \left\Vert T\right\Vert _{ev^{(2)}%
(p;q)}\left\Vert u_{1}\right\Vert \cdots\left\Vert u_{n}\right\Vert
{\textstyle\prod\limits_{k=1}^{n}}
\left(  \left\Vert a_{k}\right\Vert +\left\Vert \left(  x_{j}^{(k)}\right)
_{j=1}^{\infty}\right\Vert _{w,q}\right)
\end{align*}
and it follows that
\[
\left\Vert w\circ T\circ\left(  u_{1},\ldots,u_{n}\right)  \right\Vert
_{ev^{(2)}(p;q)}\leq\left\Vert w\right\Vert \left\Vert T\right\Vert
_{ev^{(2)}(p;q)}\left\Vert u_{1}\right\Vert \cdots\left\Vert u_{n}\right\Vert
.
\]
The other properties of multi-ideals are easily verified.
\end{proof}

In general, the ideal $(\mathcal{P}_{as(p;q)}^{n,ev},\Vert.\Vert
_{ev^{(2)}(p;q)})_{n=1}^{\infty}$ has indeed good properties (see \cite{QM}
for details). It was also proved in \cite[Proposition 4.4]{Scand} that one can
also show that
\[
\Vert P_{n}:\mathbb{K}\rightarrow\mathbb{K};P_{n}\left(  \lambda\right)
=\lambda^{n}\Vert_{ev^{(2)}(p;q)}=1
\]
for all $p\geq q\geq1$ and it is also not difficult to show that
$(\mathcal{P}_{as(p;q)}^{n,ev},\Vert.\Vert_{ev^{(2)}(p;q)})_{n=1}^{\infty}$ is
a Banach polynomial ideal. Besides, in \cite[Proposition 4.9]{Scand} it is
also proved that $(\mathcal{P}_{as(p;q)}^{n,ev},\Vert.\Vert_{ev^{(2)}%
(p;q)})_{n=1}^{\infty}$ is a global holomorphy type. The main result of this
section, Theorem \ref{plm}, shows that, contrary to what happens to the
sequence $\left(  \left(  \mathcal{P}_{as(p,q)}^{n},\left\Vert .\right\Vert
_{as(p;q)}\right)  ,\left(  \mathcal{L}_{as(p,q)}^{n},\left\Vert .\right\Vert
_{as(p;q)}\right)  \right)  _{n=1}^{\infty}$, the sequence
\[
\left(  (\mathcal{P}_{as(p;q)}^{n,ev},\Vert.\Vert_{ev^{(2)}(p;q)}%
),(\mathcal{L}_{as(p;q)}^{n,ev},\Vert.\Vert_{ev^{(2)}(p;q)})\right)
_{n=1}^{\infty}%
\]
is coherent and compatible with $\Pi_{p,q}$.

The following result is important for our purposes (note that this is a
variation of \cite[Proposition 3.5]{Scand}):

\begin{proposition}
\label{bhy}A polynomial $P$ belongs to $\mathcal{P}_{as(p;q)}^{n,ev}(^{n}E;F)$
if and only if $\overset{\vee}{P}$ belongs to $\mathcal{L}_{as(p;q)}%
^{n,ev}(^{n}E;F).$
\end{proposition}

\begin{proof}
Let us prove the nontrivial implication. Let $b_{k}\in E$ with $k=1,...,n$.
Using the Polarization Formula (see \cite{din, Mu}) we have%
\begin{align*}
&  n!\,2^{n}\left[  \overset{\vee}{P}\left(  b_{1}+x_{j}^{(1)},\ldots
,b_{n}+x_{j}^{(n)}\right)  -\overset{\vee}{P}(b_{1},\ldots,b_{n})\right] \\
&  =%
{\displaystyle\sum\limits_{\varepsilon_{i}=\pm1}}
\varepsilon_{1}\cdots\varepsilon_{n}P\left(  \varepsilon_{1}(b_{1}+x_{j}%
^{(1)})+\cdots+\varepsilon_{n}(b_{n}+x_{j}^{(n)})\right)  -%
{\displaystyle\sum\limits_{\varepsilon_{i}=\pm1}}
\varepsilon_{1}\cdots\varepsilon_{n}P\left(  \varepsilon_{1}b_{1}%
+\cdots+\varepsilon_{n}b_{n}\right) \\
&  =%
{\displaystyle\sum\limits_{\varepsilon_{i}=\pm1}}
\varepsilon_{1}\cdots\varepsilon_{n}\left[  P\left(  \left(  \varepsilon
_{1}b_{1}+\cdots+\varepsilon_{n}b_{n}\right)  +(\varepsilon_{1}x_{j}%
^{(1)}+\cdots+\varepsilon_{n}x_{j}^{(n)})\right)  -P\left(  \varepsilon
_{1}b_{1}+\cdots+\varepsilon_{n}b_{n}\right)  \right]
\end{align*}
and the result follows.
\end{proof}

\begin{proposition}
\label{peww}If $P\in\mathcal{P}_{as(p;q)}^{n,ev}\left(  ^{n}E;F\right)  $ and
$\gamma\in E^{\ast}$, then $\gamma P\in$ $\mathcal{P}_{as(p;q)}^{n+1,ev}%
\left(  ^{n+1}E;F\right)  $ and
\begin{equation}
\left\Vert \gamma P\right\Vert _{ev^{(2)}(p;q)}\leq\left\Vert \gamma
\right\Vert \left\Vert P\right\Vert _{ev^{(2)}(p;q)}. \label{fij}%
\end{equation}

\end{proposition}

\begin{proof}
Let $(x_{j})_{j=1}^{\infty}\in\ell_{q}^{u}(E)$. Note that%
\begin{align*}
&  \left(
{\displaystyle\sum_{j=1}^{\infty}}
\left\Vert \gamma(a+x_{j})P(a+x_{j})-\gamma(a)P(a)\right\Vert ^{p}\right)
^{1/p}\\
&  \leq\left\vert \gamma(a)\right\vert \left(
{\displaystyle\sum_{j=1}^{\infty}}
\left\Vert P(a+x_{j})-P(a)\right\Vert ^{p}\right)  ^{1/p}+\left(
{\displaystyle\sum_{j=1}^{\infty}}
\left\Vert \gamma(x_{j})P(a+x_{j})\right\Vert ^{p}\right)  ^{1/p}\\
&  \leq\left\Vert \gamma\right\Vert \left\Vert a\right\Vert \left\Vert
P\right\Vert _{ev^{(2)}(p;q)}\left(  \left\Vert a\right\Vert +\left\Vert
(x_{j})_{j=1}^{\infty}\right\Vert _{w,q}\right)  ^{n}+\left\Vert P\right\Vert
\left(  \sup_{j}\left\Vert a+x_{j}\right\Vert ^{n}\right)  \left(
{\displaystyle\sum_{j=1}^{\infty}}
\left\vert \gamma(x_{j})\right\vert ^{p}\right)  ^{1/p}.
\end{align*}
Since $q\leq p$,
\[
\left\Vert P\right\Vert _{ev^{(2)}(p;q)}\geq\left\Vert P\right\Vert \text{ and
}\sup_{j}\left\Vert a+x_{j}\right\Vert \leq\left(  \left\Vert a\right\Vert
+\left\Vert (x_{j})_{j=1}^{\infty}\right\Vert _{w,q}\right)  ,
\]
we have%
\begin{align*}
&  \left(
{\displaystyle\sum_{j=1}^{\infty}}
\left\Vert \gamma(a+x_{j})P(a+x_{j})-\gamma(a)P(a)\right\Vert ^{p}\right)
^{1/p}\\
&  \leq\left\Vert \gamma\right\Vert \left\Vert a\right\Vert \left\Vert
P\right\Vert _{ev^{(2)}(p;q)}\left(  \left\Vert a\right\Vert +\left\Vert
(x_{j})_{j=1}^{\infty}\right\Vert _{w,q}\right)  ^{n}\\
&  +\left\Vert P\right\Vert _{ev^{(2)}(p;q)}\left(  \left\Vert a\right\Vert
+\left\Vert (x_{j})_{j=1}^{\infty}\right\Vert _{w,q}\right)  ^{n}\left\Vert
\gamma\right\Vert \left\Vert (x_{j})_{j=1}^{\infty}\right\Vert _{w,q}\\
&  =\left\Vert \gamma\right\Vert \left\Vert P\right\Vert _{ev^{(2)}%
(p;q)}\left(  \left\Vert a\right\Vert +\left\Vert (x_{j})_{j=1}^{\infty
}\right\Vert _{w,q}\right)  ^{n+1}%
\end{align*}
and we get (\ref{fij}).
\end{proof}

\begin{proposition}
\label{pew}If $P\in\mathcal{P}_{as(p;q)}^{n+1,ev}\left(  ^{n+1}E;F\right)  $
and $a\in E$, then $P_{a}\in$ $\mathcal{P}_{as(p;q)}^{n,ev}\left(
^{n}E;F\right)  $ and
\begin{equation}
\left\Vert P_{a}\right\Vert _{ev^{(2)}(p;q)}\leq\left\Vert \overset{\vee}%
{P}\right\Vert _{ev^{(2)}(p;q)}\left\Vert a\right\Vert . \label{nhg}%
\end{equation}

\end{proposition}

\begin{proof}
Let $(x_{j})_{j=1}^{\infty}\in\ell_{q}^{u}(E)$ and $b\in E$. We just need to
note that%
\begin{align*}
\left(
{\displaystyle\sum_{j=1}^{\infty}}
\left\Vert P_{a}(b+x_{j})-P_{a}(b)\right\Vert ^{p}\right)  ^{1/p}  &  =\left(
%
{\displaystyle\sum_{j=1}^{\infty}}
\left\Vert \overset{\vee}{P}(a,\left(  b+x_{j}\right)  ^{n})-\overset{\vee}%
{P}(a,b^{n})\right\Vert ^{p}\right)  ^{1/p}\\
&  =\left(
{\displaystyle\sum_{j=1}^{\infty}}
\left\Vert \overset{\vee}{P}(a+0,\left(  b+x_{j}\right)  ^{n})-\overset{\vee
}{P}(a,b^{n})\right\Vert ^{p}\right)  ^{1/p}\\
&  \leq\left\Vert \overset{\vee}{P}\right\Vert _{ev^{(2)}(p;q)}\left(
\left\Vert a\right\Vert +\left\Vert (0)_{j=1}^{\infty}\right\Vert
_{w,q}\right)  \left(  \left\Vert b\right\Vert +\left\Vert (x_{j}%
)_{j=1}^{\infty}\right\Vert _{w,q}\right)  ^{n}.
\end{align*}

\end{proof}

\begin{proposition}
Let $i\in\left\{  1,\ldots,n+1\right\}  $. If $T\in\mathcal{L}_{as(p;q)}%
^{n+1,ev}\left(  E_{1},\ldots,E_{n+1};F\right)  $ and $a_{i}\in E_{i}$, then
$T_{a_{i}}\in\mathcal{L}_{as(p;q)}^{ev}\left(  E_{1},\ldots,E_{i-1}%
,E_{i+1},\ldots,E_{n+1};F\right)  $ and
\begin{equation}
\left\Vert T_{a_{i}}\right\Vert _{ev^{\left(  2\right)  }\left(  p;q\right)
}\leq\left\Vert T\right\Vert _{ev^{\left(  2\right)  }\left(  p;q\right)
}\left\Vert a_{i}\right\Vert . \label{mul}%
\end{equation}

\end{proposition}

\begin{proof}
Let $\left(  x_{j}^{\left(  k\right)  }\right)  _{j=1}^{\infty}\in l_{q}%
^{u}\left(  E_{k}\right)  ~$and $a_{j}\in E_{j}$ for all $j\neq i$. The proof
follows from the inequalities
\begin{align*}
&  \left(  \sum_{j=1}^{\infty}\left\Vert
\begin{array}
[c]{c}%
T_{a_{i}}\left(  a_{1}+x_{j}^{\left(  1\right)  },\ldots,a_{i-1}%
+x_{j}^{\left(  i-1\right)  },a_{i+1}+x_{j}^{\left(  i+1\right)  }%
,\ldots,a_{n+1}+x_{j}^{\left(  n+1\right)  }\right) \\
-T_{a_{i}}\left(  a_{1},\ldots,a_{i-1},a_{i+1},\ldots,a_{n+1}\right)
\end{array}
\right\Vert ^{p}\right)  ^{1/p}\\
&  =\left(  \sum_{j=1}^{\infty}\left\Vert
\begin{array}
[c]{c}%
T\left(  a_{1}+x_{j}^{\left(  1\right)  },\ldots,a_{i-1}+x_{j}^{\left(
i-1\right)  },a_{i}+0,a_{i+1}+x_{j}^{\left(  i+1\right)  },\ldots
,a_{n+1}+x_{j}^{\left(  n+1\right)  }\right) \\
-T\left(  a_{1},\ldots,a_{n+1}\right)
\end{array}
\right\Vert ^{p}\right)  ^{1/p}\\
&  \leq\left\Vert T\right\Vert _{ev^{\left(  2\right)  }\left(  p;q\right)
}\left\Vert a_{i}\right\Vert
{\displaystyle\prod\limits_{\substack{k=1\\k\neq i}}^{n+1}}
\left(  \left\Vert a_{k}\right\Vert +\left\Vert \left(  x_{j}^{\left(
k\right)  }\right)  _{j=1}^{\infty}\right\Vert _{w,q}\right)  .
\end{align*}

\end{proof}

\begin{proposition}
\label{mult1}If $T\in\mathcal{L}_{as(p;q)}^{n,ev}\left(  E_{1},\ldots
,E_{n};F\right)  $ and $\gamma\in E_{k+1}^{\ast}$, then
\begin{equation}
\gamma T\in\mathcal{L}_{as(p;q)}^{n+1,ev}\left(  E_{1},\ldots,E_{n+1}%
;F\right)  \text{ and }\left\Vert \gamma T\right\Vert _{ev^{\left(  2\right)
}\left(  p;q\right)  }\leq\left\Vert \gamma\right\Vert \left\Vert T\right\Vert
_{ev^{\left(  2\right)  }\left(  p;q\right)  }. \label{mta}%
\end{equation}

\end{proposition}

\begin{proof}
Let $\left(  x_{j}^{\left(  k\right)  }\right)  _{j=1}^{\infty}\in l_{q}%
^{u}\left(  E_{k}\right)  ~$and $\left(  a_{1},\ldots,a_{n+1}\right)  \in
E_{1}\times\cdots\times E_{n+1}$. Then,
\begin{align*}
&  \left(  \sum_{j=1}^{\infty}\left\Vert \gamma\left(  a_{n+1}+x_{j}^{\left(
n+1\right)  }\right)  T\left(  a_{1}+x_{j}^{\left(  1\right)  },\ldots
,a_{n}+x_{j}^{\left(  n\right)  }\right)  -\gamma\left(  a_{n+1}\right)
T\left(  a_{1},\ldots,a_{n}\right)  \right\Vert ^{p}\right)  ^{1/p}\\
&  \leq\left\vert \gamma\left(  a_{n+1}\right)  \right\vert \left(  \sum
_{j=1}^{\infty}\left\Vert T\left(  a_{1}+x_{j}^{\left(  1\right)  }%
,\ldots,a_{n}+x_{j}^{\left(  n\right)  }\right)  -T\left(  a_{1},\ldots
,a_{n}\right)  \right\Vert ^{p}\right)  ^{1/p}\\
&  +\left(  \sum_{j=1}^{\infty}\left\Vert \gamma\left(  x_{j}^{\left(
n+1\right)  }\right)  T\left(  a_{1}+x_{j}^{\left(  1\right)  },\ldots
,a_{n}+x_{j}^{\left(  n\right)  }\right)  \right\Vert ^{p}\right)  ^{1/p}\\
&  \leq\left\vert \gamma\left(  a_{n+1}\right)  \right\vert \left(  \sum
_{j=1}^{\infty}\left\Vert T\left(  a_{1}+x_{j}^{\left(  1\right)  }%
,\ldots,a_{n}+x_{j}^{\left(  n\right)  }\right)  -T\left(  a_{1},\ldots
,a_{n}\right)  \right\Vert ^{p}\right)  ^{1/p}\\
&  +\left\Vert T\right\Vert \sup_{j}%
{\displaystyle\prod\limits_{k=1}^{n}}
\left(  \left\Vert a_{k}+x_{j}^{\left(  k\right)  }\right\Vert \right)
\left(  \sum_{j=1}^{\infty}\left\vert \gamma\left(  x_{j}^{\left(  n+1\right)
}\right)  \right\vert ^{p}\right)  ^{1/p}.
\end{align*}
Using the same arguments of Proposition \ref{peww} we have
\begin{align*}
&  \left(  \sum_{j=1}^{\infty}\left\Vert \gamma\left(  a_{n+1}+x_{j}^{\left(
n+1\right)  }\right)  T\left(  a_{1}+x_{j}^{\left(  1\right)  },\ldots
,a_{n}+x_{j}^{\left(  n\right)  }\right)  -\gamma\left(  a_{n+1}\right)
T\left(  a_{1},\ldots,a_{n}\right)  \right\Vert ^{p}\right)  ^{1/p}\\
&  \leq\left\Vert \gamma\right\Vert \left\Vert a_{n+1}\right\Vert \left\Vert
T\right\Vert _{ev^{\left(  2\right)  }\left(  p;q\right)  }%
{\displaystyle\prod\limits_{k=1}^{n}}
\left(  \left\Vert a_{k}\right\Vert +\left\Vert \left(  x_{j}^{\left(
k\right)  }\right)  _{j=1}^{\infty}\right\Vert _{w,q}\right) \\
&  +\left\Vert T\right\Vert _{ev^{\left(  2\right)  }\left(  p;q\right)  }%
{\displaystyle\prod\limits_{k=1}^{n}}
\left(  \left\Vert a_{k}\right\Vert +\left\Vert \left(  x_{j}^{\left(
k\right)  }\right)  _{j=1}^{\infty}\right\Vert _{w,q}\right)  \left\Vert
\gamma\right\Vert \left\Vert \left(  x_{j}^{\left(  n+1\right)  }\right)
_{j=1}^{\infty}\right\Vert _{w,q}\\
&  =\left\Vert \gamma\right\Vert \left\Vert T\right\Vert _{ev^{\left(
2\right)  }\left(  p;q\right)  }%
{\displaystyle\prod\limits_{k=1}^{n+1}}
\left(  \left\Vert a_{k}\right\Vert +\left\Vert \left(  x_{j}^{\left(
k\right)  }\right)  _{j=1}^{\infty}\right\Vert _{w,q}\right)
\end{align*}
and the proof is done.
\end{proof}

The coherence and compatibility of $\left(  (\mathcal{P}_{as(p;q)}%
^{n,ev},\Vert.\Vert_{ev^{(2)}(p;q)}),(\mathcal{L}_{as(p;q)}^{n,ev},\Vert
.\Vert_{ev^{(2)}(p;q)})\right)  _{n=1}^{\infty}$ is a consequence of the
previous results and Remark \ref{yv}:

\begin{theorem}
\label{plm}The sequence $\left(  (\mathcal{P}_{as(p;q)}^{n,ev},\Vert
.\Vert_{ev^{(2)}(p;q)}),(\mathcal{L}_{as(p;q)}^{n,ev},\Vert.\Vert
_{ev^{(2)}(p;q)})\right)  _{n=1}^{\infty}$ is coherent and compatible with
$\Pi_{p,q}$.
\end{theorem}

\section{Strongly summing multilinear operators and polynomials\label{jt}}

The multi-ideal of strongly $p$-summing multilinear operators is one of the
classes that best inherits the spirit of the ideal of absolutely $p$-summing
linear operators (for papers comparing the different nonlinear extensions of
absolutely summing operators we refer to \cite{rm, QM, davidarchiv}).

If $p\geq1,$ $T\in\mathcal{L}(E_{1},...,E_{n};F)$ is strongly $p$-summing
($T\in\mathcal{L}\Pi_{p}^{n,\text{str}}(E_{1},...,E_{n};F)$) if there exists a
constant $C\geq0$ such that
\begin{equation}
\left(  \sum\limits_{j=1}^{m}\parallel T(x_{j}^{(1)},...,x_{j}^{(n)}%
)\parallel^{p}\right)  ^{1/p}\leq C\left(  \underset{\phi\in B_{\mathcal{L}%
(E_{1},...,E_{n};\mathbb{K})}}{\sup}\sum\limits_{j=1}^{m}\mid\phi(x_{j}%
^{(1)},...,x_{j}^{(n)})\mid^{p}\right)  ^{1/p}. \label{rp}%
\end{equation}
for all $m\in\mathbb{N}$, $x_{j}^{(l)}\in E_{l}$ with $l=1,...,n$ and
$j=1,...,m.$ The infimum of all $C\geq0$ satisfying (\ref{rp}) defines a
complete norm, denoted by $\left\Vert .\right\Vert _{\mathcal{L}\Pi
_{p,q}^{n,\text{str}}},$ on the space $\mathcal{L}\Pi_{p,q}^{n,\text{str}%
}(E_{1},...,E_{n};F)$.

This concept is due to V. Dimant \cite{dimant}. In the same paper the author
proposes a definition for the polynomial case, but as mentioned in
\cite{CDM09} this concept does not generate a polynomial ideal compatible with
$\Pi_{p}$.

It is easy to prove that the ideal of strongly $p$-summing multilinear
operators is closed under differentiation, closed for scalar multiplication.
Besides, for this class we have a Grothendieck-type theorem and a
Pietsch-Domination type theorem:

\begin{theorem}
(\cite{dimant}) If $n\geq2$, then
\[
\mathcal{L}(^{n}\ell_{1};\ell_{2})=\mathcal{L}\Pi_{1,1}^{n,\text{str}}%
(^{n}\ell_{1};\ell_{2}).
\]

\end{theorem}

\begin{theorem}
(\cite{dimant}) $T\in\mathcal{L}\left(  E_{1},...,E_{n};F\right)  $ is
strongly $p$-summing if, and only if, there are a probability measure $\mu$ on
$B_{(E_{1}\otimes_{\pi}\cdots\otimes_{\pi}E_{n})^{\ast}}$, with the weak-star
topology, and a constant $C\geq0$ so that
\begin{equation}
\left\Vert T\left(  x_{1},...,x_{n}\right)  \right\Vert \leq C\left(
\int_{B_{(E_{1}\otimes_{\pi}\cdots\otimes_{\pi}E_{n})^{\ast}}}\left\vert
\varphi\left(  x_{1}\otimes\cdots\otimes x_{n}\right)  \right\vert ^{p}%
d\mu\left(  \varphi\right)  \right)  ^{1/p} \label{7out08c}%
\end{equation}
for all $(x_{1},...,x_{n})\in E_{1}\times\cdots\times E_{n}.$
\end{theorem}

As a consequence, there is an Inclusion Theorem (if $p\leq q$ then
$\mathcal{L}\Pi_{p}^{n,\text{str}}\subset\mathcal{L}\Pi_{q}^{n,\text{str}}$).
It is worth mentioning that even the fashionable multi-ideal of multiple
summing multilinear operators (see \cite{Collect}) does not have all these
properties (for example, P\'{e}rez-Garc\'{\i}a \cite{d4} proved that the
inclusion theorem is not valid, in general). In the recent paper \cite{QM} a
notion of \textquotedblleft desired generalisation of absolutely summing
operators\textquotedblright\ to the multilinear setting was discussed and the
class of strongly $p$-summing multilinear operators seemed to be one of the
\textquotedblleft closest to perfection\textquotedblright.

We will choose a concept of strongly $p$-summing polynomials different from
the one from \cite{CDM09, dimant}; we will consider that $P\in\mathcal{P}%
(^{n}E;F)$ is strongly $p$-summing (notation $\mathcal{P}\Pi_{p}%
^{n,\text{str}}$) if $\overset{\vee}{P}$ belongs to $\mathcal{L}\Pi
_{p}^{n,\text{str}}(^{n}E;F)$ and
\[
\left\Vert P\right\Vert _{\mathcal{P}\Pi_{p}^{n,\text{str}}}:=\left\Vert
\overset{\vee}{P}\right\Vert _{\mathcal{L}\Pi_{p}^{n,\text{str}}}.
\]
So, $\left(  \mathcal{P}\Pi_{p}^{n,\text{str}},\left\Vert .\right\Vert
_{\mathcal{P}\Pi_{p}^{n,\text{str}}}\right)  _{n=1}^{\infty}$ is a Banach
polynomial ideal and it is easy to prove that the sequence $\left(
\mathcal{P}\Pi_{p}^{n,\text{str}},\left\Vert .\right\Vert _{\mathcal{P}\Pi
_{p}^{n,\text{str}}}\right)  _{n=1}^{\infty}$ is a global holomorphy type. As
we will see in the forthcoming Theorem \ref{joi}, the sequence $\left(
\left(  \mathcal{P}\Pi_{p}^{n,\text{str}},\left\Vert .\right\Vert
_{\mathcal{P}\Pi_{p}^{n,\text{str}}}\right)  ,\left(  \mathcal{L}\Pi
_{p}^{n,\text{str}},\left\Vert .\right\Vert _{\mathcal{L}\Pi_{p}%
^{n,\text{str}}}\right)  \right)  _{n=1}^{\infty}$ is coherent and compatible
with $\Pi_{p}$, and this result is quite adequate in view of the very good
properties of the ideal of strongly summing multilinear mappings. For its
proof we need to prove four simple propositions:

\begin{proposition}
\label{coerente2} Let $k\in\left\{  1,\ldots,n+1\right\}  $. If $T\in
\mathcal{L}\Pi_{p}^{n+1,\text{str}}\left(  E_{1},\ldots,E_{n+1};F\right)  $
and $a_{k}\in E_{k}$, then
\[
T_{a_{k}}\in\mathcal{L}\Pi_{p}^{n,\text{str}}\left(  E_{1},\ldots
,E_{n};F\right)  \text{ and }\left\Vert T_{a_{k}}\right\Vert _{\mathcal{L}%
\Pi_{p}^{n,\text{str}}}\leq\left\Vert a_{k}\right\Vert \left\Vert T\right\Vert
_{\mathcal{L}\Pi_{p}^{n+1,\text{str}}}.
\]

\end{proposition}

\begin{proof}
The case $a_{k}=0$ is immediate. Let us suppose $a_{k}\neq0.$ We just have to
note that%
\begin{align*}
&  \left(  \sum_{j=1}^{m}\left\Vert T_{a_{k}}\left(  x_{j}^{\left(  1\right)
},\ldots,x_{j}^{\left(  k-1\right)  },x_{j}^{\left(  k+1\right)  }%
,\ldots,x_{j}^{\left(  n\right)  }\right)  \right\Vert ^{p}\right)  ^{1/p}\\
&  =\left(  \sum_{j=1}^{m}\left\Vert T\left(  x_{j}^{\left(  1\right)
},\ldots,x_{j}^{\left(  k-1\right)  },a_{k},x_{j}^{\left(  k+1\right)
},\ldots,x_{j}^{\left(  n\right)  }\right)  \right\Vert ^{p}\right)  ^{1/p}\\
&  \leq\left\Vert T\right\Vert _{\mathcal{L}\Pi_{p}^{n+1,\text{str}}%
}\left\Vert a_{k}\right\Vert \sup_{\varphi\in B_{\mathcal{L}(E_{1}%
,\ldots,E_{n};\mathbb{K})}}\left(  \sum_{j=1}^{m}\left\vert \varphi\left(
x_{j}^{\left(  1\right)  },\ldots,x_{j}^{\left(  k-1\right)  },\frac{a_{k}%
}{\left\Vert a_{k}\right\Vert },x_{j}^{\left(  k+1\right)  },\ldots
,x_{j}^{\left(  n\right)  }\right)  \right\vert ^{p}\right)  ^{1/p}\\
&  \leq\left\Vert T\right\Vert _{\mathcal{L}\Pi_{p}^{n+1,\text{str}}%
}\left\Vert a_{k}\right\Vert \sup_{\varphi\in B_{\mathcal{L}(E_{1}%
,\ldots,E_{k-1},E_{k+1},\ldots,E_{n+1};\mathbb{K})}}\left(  \sum_{j=1}%
^{m}\left\vert \varphi\left(  x_{j}^{\left(  1\right)  },\ldots,x_{j}^{\left(
k-1\right)  },x_{j}^{\left(  k+1\right)  },\ldots,x_{j}^{\left(  n\right)
}\right)  \right\vert ^{p}\right)  ^{1/p}.
\end{align*}

\end{proof}

\begin{proposition}
\label{coerente4} If $T\in\mathcal{L}\Pi_{p}^{n,\text{str}}\left(
E_{1},\ldots,E_{n};F\right)  $ and $\gamma\in E_{n+1}^{\ast}$, then
\[
\gamma T\in\mathcal{L}\Pi_{p}^{n+1,\text{str}}\left(  E_{1},\ldots
,E_{n+1};F\right)  \text{ and }\left\Vert \gamma T\right\Vert _{\mathcal{L}%
\Pi_{p}^{n+1,\text{str}}}\leq\left\Vert \gamma\right\Vert \left\Vert
T\right\Vert _{\mathcal{L}\Pi_{p}^{n,\text{str}}}.
\]

\end{proposition}

\begin{proof}
Since $T\in\mathcal{L}\Pi_{p}^{n,\text{str}}\left(  E_{1},\ldots
,E_{n};F\right)  $, we have
\begin{align*}
&  \left(  \sum\limits_{j=1}^{m}\parallel T(x_{j}^{(1)},...,x_{j}^{(n)}%
)\gamma\left(  x_{j}^{(n+1}\right)  \parallel^{p}\right)  ^{1/p}\\
&  \leq\left\Vert T\right\Vert _{\mathcal{L}\Pi_{p}^{n,\text{str}}}\left(
\underset{\phi\in B_{\mathcal{L}(E_{1},...,E_{n};\mathbb{K})}}{\sup}%
\sum\limits_{j=1}^{m}\mid\phi(x_{j}^{(1)},...,x_{j}^{(n)})\gamma\left(
x_{j}^{(n+1}\right)  \mid^{p}\right)  ^{1/p}\\
&  \leq\left\Vert \gamma\right\Vert \left\Vert T\right\Vert _{\mathcal{L}%
\Pi_{p}^{n,\text{str}}}\left(  \underset{\psi\in B_{\mathcal{L}(E_{1}%
,...,E_{n+1};\mathbb{K})}}{\sup}\sum\limits_{j=1}^{m}\mid\psi(x_{j}%
^{(1)},...,x_{j}^{(n+1)})\mid^{p}\right)  ^{1/p}%
\end{align*}
and the proof is done.
\end{proof}

\begin{proposition}
\label{coerente1} If $P\in\mathcal{P}\Pi_{p}^{n+1,\text{str}}\left(
^{n+1}E;F\right)  $ and $a\in E$, then $P_{a}$ belongs to $\mathcal{P}\Pi
_{p}^{n,\text{str}}\left(  ^{n}E;F\right)  $ and
\[
\left\Vert P_{a}\right\Vert _{\mathcal{P}\Pi_{p}^{n,\text{str}}}\leq\left\Vert
a\right\Vert \left\Vert \overset{\vee}{P}\right\Vert _{\mathcal{L}\Pi
_{p}^{n+1,\text{str}}}.
\]

\end{proposition}

\begin{proof}
Since $\overset{\vee}{P}\in\mathcal{L}\Pi_{p}^{n+1,\text{str}}\left(
^{n}E;F\right)  ,$ from Proposition \ref{coerente2} we have
\[
\overset{\vee}{P}_{a}\in\mathcal{L}\Pi_{p}^{n,\text{str}}\left(
^{n}E;F\right)
\]
and
\[
\left\Vert \overset{\vee}{P}_{a}\right\Vert _{\mathcal{L}\Pi_{p}%
^{n,\text{str}}}\leq\left\Vert a\right\Vert \left\Vert \overset{\vee}%
{P}\right\Vert _{\mathcal{L}\Pi_{p}^{n+1,\text{str}}}.
\]
Hence
\[
\left\Vert P_{a}\right\Vert _{\mathcal{P}\Pi_{p}^{n,\text{str}}}=\left\Vert
\left(  P_{a}\right)  ^{\vee}\right\Vert _{\mathcal{L}\Pi_{p}^{n,\text{str}}%
}=\left\Vert \overset{\vee}{P}_{a}\right\Vert _{\mathcal{L}\Pi_{p}%
^{n,\text{str}}}\leq\left\Vert a\right\Vert \left\Vert \overset{\vee}%
{P}\right\Vert _{\mathcal{L}\Pi_{p}^{n+1,\text{str}}}.
\]

\end{proof}

\begin{proposition}
\label{coerente3} If $P\in\mathcal{P}\Pi_{p}^{n,\text{str}}\left(
^{n}E;F\right)  $ and $\gamma\in E^{\ast}$, then $\gamma P$ belongs to
$\mathcal{P}\Pi_{p}^{n+1,\text{str}}\left(  ^{n+1}E;F\right)  $ and
\[
\left\Vert \gamma P\right\Vert _{\mathcal{P}\Pi_{p}^{n+1,\text{str}}}%
\leq\left\Vert \gamma\right\Vert \left\Vert \overset{\vee}{P}\right\Vert
_{\mathcal{L}\Pi_{p}^{n,\text{str}}}.
\]

\end{proposition}

\begin{proof}
Since%
\[
\left(  \gamma P\right)  ^{\vee}(x_{1},...,x_{n+1})=\frac{\gamma
(x_{1})\overset{\vee}{P}(x_{2},...,x_{n+1})+\cdots+\gamma(x_{n+1}%
)\overset{\vee}{P}(x_{1},...,x_{n})}{n+1},
\]
from Proposition \ref{coerente4} we have
\[
\left\Vert \gamma P\right\Vert _{\mathcal{P}\Pi_{p}^{n+1,\text{str}}%
}=\left\Vert \left(  \gamma P\right)  ^{\vee}\right\Vert _{\mathcal{L}\Pi
_{p}^{n+1,\text{str}}}\leq\left\Vert \gamma\overset{\vee}{P}\right\Vert
_{\mathcal{L}\Pi_{p}^{n+1,\text{str}}}\leq\left\Vert \gamma\right\Vert
\left\Vert \overset{\vee}{P}\right\Vert _{\mathcal{L}\Pi_{p}^{n,\text{str}}}.
\]

\end{proof}

As in the previous sections, the following theorem is a consequence of the
previous propositions and Remark \ref{yv}:

\begin{theorem}
\label{joi}The sequence $\left(  \left(  \mathcal{P}\Pi_{p}^{n,\text{str}%
},\left\Vert .\right\Vert _{\mathcal{P}\Pi_{p}^{n,\text{str}}}\right)
,\left(  \mathcal{L}\Pi_{p}^{n,\text{str}},\left\Vert .\right\Vert
_{\mathcal{L}\Pi_{p}^{n,\text{str}}}\right)  \right)  _{n=1}^{\infty}$ is
coherent and compatible with $\Pi_{p}.$
\end{theorem}

\section{Multiple summing polynomials and multilinear operators\label{jt2}}

If $1\leq q\leq p<\infty$ and $n$ is a positive integer$,$ an $n$-linear
operator $T:E_{1}\times\cdots\times E_{n}\rightarrow F$ is multiple\emph{
}$(p;q)$-summing ($T\in\mathcal{L}\Pi_{p,q}^{n,\text{mult}}(E_{1}%
,...,E_{n};F)$) if there exists $C>0$ such that
\begin{equation}
\left(  \sum_{j_{1},...,j_{n}=1}^{\infty}\Vert T(x_{j_{1}}^{(1)},...,x_{j_{n}%
}^{(n)})\Vert^{p}\right)  ^{1/p}\leq C\prod\limits_{k=1}^{n}\Vert(x_{j}%
^{(k)})_{j=1}^{\infty}\Vert_{w,q}\text{ } \label{jup2}%
\end{equation}
for all $(x_{j}^{(k)})_{j=1}^{\infty}\in\ell_{q}^{w}(E_{k})$, $k=1,...,n$. The
infimum of all $C\geq0$ satisfying (\ref{jup2}) defines a complete norm,
denoted by $\left\Vert .\right\Vert _{\mathcal{L}\Pi_{p,q}^{n,\text{mult}}},$
on the space $\mathcal{L}\Pi_{p,q}^{n,\text{mult}}(E_{1},...,E_{n};F)$.

This class was introduced by M.C. Matos \cite{Collect} and, independently, by
F. Bombal, D. P\'{e}rez-Garc\'{\i}a and I. Villanueva \cite{bombal} and was
investigated by different authors in recent years (see, for example
\cite{botp, dddd, d4, popa}).

The ideal of multiple summing polynomials is defined as in Definition \ref{or}
and denoted by $\mathcal{P}\Pi_{p,q}^{n,\text{mult}}$ (and the norm is denoted
by $\left\Vert .\right\Vert _{\mathcal{P}\Pi_{p,q}^{n,\text{mult}}}$). By
using essentially the same arguments from the previous section we can prove:

\begin{theorem}
The sequence $\left(  \left(  \mathcal{P}\Pi_{p,q}^{n,\text{mult}},\left\Vert
.\right\Vert _{\mathcal{P}\Pi_{p,q}^{n,\text{mult}}}\right)  ,\left(
\mathcal{L}\Pi_{p,q}^{n,\text{mult}},\left\Vert .\right\Vert _{\mathcal{L}%
\Pi_{p,q}^{n,\text{mult}}}\right)  \right)  _{n=1}^{\infty}$ is coherent and
compatible with $\Pi_{p,q}.$
\end{theorem}

We finish this short section with an illustrative example on how the concepts of
coherence and compatibility work well together.

\begin{example}
\label{ffiiu}For all $n$ let $\mathcal{U}_{1}=\mathcal{M}_{1}=\Pi_{1,1}$ and
consider the artificial sequence $\left(  \mathcal{U}_{n},\mathcal{M}%
_{n}\right)  _{n=1}^{\infty}$ defined by%
\begin{align*}
\left(  \mathcal{U}_{2n},\mathcal{M}_{2n}\right)  _{n=1}^{\infty}  &  =\left(
\left(  \mathcal{P}\Pi_{1,1}^{2n,\text{mult}},\left\Vert .\right\Vert
_{\mathcal{P}\Pi_{1,1}^{2n,\text{mult}}}\right)  ,\left(  \mathcal{L}\Pi
_{1,1}^{2n,\text{mult}},\left\Vert .\right\Vert _{\mathcal{L}\Pi
_{1,1}^{2,\text{mult}}}\right)  \right)  _{n=1}^{\infty}\\
\left(  \mathcal{U}_{2n+1},\mathcal{M}_{2n+1}\right)  _{n=1}^{\infty}  &
=\left(  \left(  \mathcal{P}\Pi_{1,1}^{2n+1,\text{str}},\left\Vert
.\right\Vert _{\mathcal{P}\Pi_{1,1}^{2n+1,\text{str}}}\right)  ,\left(
\mathcal{L}\Pi_{1,1}^{2n+1,\text{str}},\left\Vert .\right\Vert _{\mathcal{L}%
\Pi_{1,1}^{2n+1,\text{str}}}\right)  \right)  _{n=1}^{\infty}.
\end{align*}
From our previous results it is easy to see that the sequence $\left(
\mathcal{U}_{n},\mathcal{M}_{n}\right)  _{n=1}^{\infty}$ is compatible with
$\Pi_{1,1}.$ On the other hand it is also not difficult to show that $\left(
\mathcal{U}_{n},\mathcal{M}_{n}\right)  _{n=1}^{\infty}$ is not coherent. So
the concepts of coherent and compatible pairs, together, besides filtering
sequences that keep the spirit of a given operator ideal, also seem to be an
adequate method of avoiding artificial constructions.
\end{example}

\section{Strongly coherent and compatible sequences}

An apparently stronger notion of coherence and compatibility can be considered
if we replace (CP2) and (CH2) by (respectively)

\bigskip

(CP2$^{\ast}$) There is a constant $\alpha_{2}>0$ so that if $P\in
\mathcal{U}_{n}\left(  ^{n}E;F\right)  $ and $a\in E$, then $P_{a^{n-1}}%
\in\mathcal{U}\left(  E;F\right)  $ and%
\[
\left\Vert P_{a^{n-1}}\right\Vert _{\mathcal{U}}\leq\alpha_{2}\left\Vert
P\right\Vert _{\mathcal{U}_{n}}\left\Vert a\right\Vert ^{n-1}.
\]

(CH2$^{\ast}$) There is a constant $\beta_{2}>0$ so that if $P\in
\mathcal{U}_{k+1}\left(  ^{k+1}E;F\right)  ,$ $a\in E$, then $P_{a}$ belongs
to $\mathcal{U}_{k}\left(  ^{k}E;F\right)  $ and
\[
\left\Vert P_{a}\right\Vert _{\mathcal{U}_{k}}\leq\beta_{2}\left\Vert
P\right\Vert _{\mathcal{U}_{k+1}}\left\Vert a\right\Vert .
\]

\bigskip

This approach is closer to the original concepts from \cite{CDM09} for
polynomial ideals. For the cases investigated in the Sections \ref{jt0},
\ref{jt} and \ref{jt2} there is no difference between the concepts since%
\[
\left\Vert P\right\Vert _{\mathcal{U}_{n}}=\left\Vert \overset{\vee}%
{P}\right\Vert _{\mathcal{M}_{n}}.
\]
The case of the \textquotedblleft canonical\textquotedblright\ pairs
$(\mathcal{P}_{k},\mathcal{L}_{k})_{k=1}^{\infty}$ (composed by the ideals of
continuous $n$-homogeneous polynomials and continuous $n$-linear operators,
with the $\sup$ norm) illustrates that the notion of strongly coherent pairs
has important restrictions. In fact, from \cite[Proposition 8.5]{BBJP} we know
that (CH2$^{\ast}$) is not valid for this case when dealing with real scalars.

However, for the case of $r$-dominated multilinear operators and polynomials
we remark that strong coherence and compatibility holds.

From now on $n\in\mathbb{N}$ and $r\in\left[  n,\infty\right]  $.

\begin{definition}
A multilinear operator $T:E_{1}\times\cdots\times E_{n}\rightarrow F$ is
$r$-dominated (in this case we write $T\in\mathcal{L}_{d,r}^{n}\left(
E_{1},...,E_{n};F\right)  $) if there exists a constant $C>0$ such that%
\[
\left(  \sum_{j=1}^{m}\left\Vert T\left(  x_{j}^{\left(  1\right)  }%
,\ldots,x_{j}^{\left(  n\right)  }\right)  \right\Vert ^{r/n}\right)
^{n/r}\leq C%
{\displaystyle\prod\limits_{j=1}^{n}}
\left\Vert \left(  x_{i}^{\left(  j\right)  }\right)  _{i=1}^{m}\right\Vert
_{w,r}%
\]
\ for all $m\in\mathbb{N}$ and $x_{1}^{\left(  j\right)  },\ldots
,x_{m}^{\left(  j\right)  }\in E_{j}$. The smallest such $C$ is denoted by
$\left\Vert T\right\Vert _{d,r}$. It is well-known that $\left(
\mathcal{L}_{d,r},\left\Vert \cdot\right\Vert _{d,r}\right)  $ is a Banach
ideal of multilinear mappings (recall that $n\leq r$).
\end{definition}

The terminology \textquotedblleft dominated\textquotedblright\ is motivated by
the Pietsch Domination Theorem:

\begin{theorem}
[Pietsch, Geiss, 1985](\cite{geisse}) $T\in\mathcal{L}(E_{1},...,E_{n};F)$ is
$r$-dominated if and only if there exist $C\geq0$ and probability measures
$\mu_{j}$ on the Borel $\sigma$-algebras of $B_{E_{j}^{^{\ast}}}$ endowed with
the weak star topologies such that
\[
\left\Vert T\left(  x_{1},...,x_{n}\right)  \right\Vert \leq C\prod
\limits_{j=1}^{n}\left(  \int_{B_{E_{j}^{\ast}}}\left\vert \varphi\left(
x_{j}\right)  \right\vert ^{p}d\mu_{j}\left(  \varphi\right)  \right)  ^{1/p}%
\]
for all $x_{j}\in E_{j}$ and $j=1,...,n$. Moreover, the infimum of the $C$
that satisfy the inequality above is precisely $\left\Vert T\right\Vert
_{d,r}$.
\end{theorem}

For recent generalisations of the Pietsch Domination Theorem and related
results we mention \cite{jmaa1, ps2, london}. The concept of $r$-dominated
polynomial is similar, \textit{mutatis mutandis}, to the notion for
multilinear operators:

\begin{definition}
An $n$-homogeneous polynomial $P:E\rightarrow F$ is $r$-dominated (in this
case we write $P\in\mathcal{P}_{d,r}^{n}\left(  ^{n}E;F\right)  $), if there
is a constant $C>0$ such that
\[
\left(  \sum_{i=1}^{m}\left\Vert P\left(  x_{i}\right)  \right\Vert
^{r/n}\right)  ^{n/r}\leq C\left\Vert \left(  x_{i}\right)  _{i=1}%
^{m}\right\Vert _{w,r}^{k}%
\]
for all $m\in\mathbb{N}$ and $x_{1},\ldots,x_{m}\in E$. The smallest such $C$
is denoted by $\left\Vert P\right\Vert _{d,r}$. It is well-known that $\left(
\mathcal{P}_{d,r},\left\Vert \cdot\right\Vert _{d,r}\right)  $ is a Banach
ideal of polynomials.
\end{definition}

The next result is folklore:

\begin{proposition}
\label{Dominated1} A polynomial $P$ belongs to $\mathcal{P}_{d,r}^{n}\left(
^{n}E;F\right)  $ if and only if $\overset{\vee}{P}$ belongs to $\mathcal{L}%
_{d,r}^{n}\left(  ^{n}E;F\right)  $.
\end{proposition}

From \cite{CDM09} we know that (CP2$^{\ast}$), (CH2$^{\ast}$), (CP4) and (CH4)
are valid for $\left(  \mathcal{P}_{d,r}^{n},\left\Vert \cdot\right\Vert
_{d,r}\right)  _{n=1}^{\infty}.$

The other properties are easily verified. For example, let us check (CH3):

\begin{proposition}
\label{Dominated7} If $T\in\mathcal{L}_{d,r}^{k}\left(  E_{1},\ldots
,E_{k};F\right)  $ and $\gamma\in E_{k+1}^{\ast}$, $\gamma T\in\mathcal{L}%
_{d,r}^{k+1}\left(  E_{1},\ldots,E_{k+1};F\right)  $ then
\[
\left\Vert \gamma T\right\Vert _{d,r}\leq\left\Vert \mathcal{\gamma
}\right\Vert \left\Vert T\right\Vert _{d,r}.
\]

\end{proposition}

\begin{proof}
From the Pietsch Domination Theorem there are Borel probability measures
$\mu_{1},...,\mu_{k}$ so that
\[
\left\Vert T(x_{1},...,x_{k})\right\Vert \leq\left\Vert T\right\Vert _{d,r}%
{\textstyle\prod\limits_{i=1}^{k}}
\left(
{\textstyle\int\nolimits_{B_{E_{i}^{\ast}}}}
\left\vert \varphi_{i}(x_{i})\right\vert ^{r}d\mu_{i}\right)  ^{1/r}.
\]
By considering the Dirac measure $\delta_{\gamma/\left\Vert \gamma\right\Vert
}$ we have%
\begin{align*}
&  \left\Vert \gamma(x_{k+1})T(x_{1},...,x_{k})\right\Vert \\
&  \leq\left\Vert \gamma\right\Vert \left\Vert T\right\Vert _{d,r}\left\vert
\frac{\gamma}{\left\Vert \gamma\right\Vert }(x_{k+1})\right\vert
{\textstyle\prod\limits_{i=1}^{k}}
\left(
{\textstyle\int\nolimits_{B_{E_{i}^{\ast}}}}
\left\vert \varphi_{i}(x_{i})\right\vert ^{r}d\mu_{i}\right)  ^{1/r}\\
&  =\left\Vert \gamma\right\Vert \left\Vert T\right\Vert _{d,r}\left(
{\textstyle\int\nolimits_{B_{E_{k+1}^{\ast}}}}
\left\vert \varphi_{k+1}(x_{k+1})\right\vert ^{r}d\delta_{\gamma/\left\Vert
\gamma\right\Vert }\right)  ^{1/r}%
{\textstyle\prod\limits_{i=1}^{k}}
\left(
{\textstyle\int\nolimits_{B_{E_{i}^{\ast}}}}
\left\vert \varphi_{i}(x_{i})\right\vert ^{r}d\mu_{i}\right)  ^{1/r}%
\end{align*}
and again the Pietsch Domination Theorem asserts that%
\[
\left\Vert \gamma T\right\Vert _{d,r}\leq\left\Vert \gamma\right\Vert
\left\Vert T\right\Vert _{d,r}.
\]

\end{proof}

So, we have:

\begin{proposition}
The sequence $\left(  \left(  \mathcal{P}_{d,r}^{n},\left\Vert .\right\Vert
_{d,r}\right)  ,\left(  \mathcal{L}_{d,r}^{n},\left\Vert .\right\Vert
_{d,r}\right)  \right)  _{n=1}^{\infty}$ is strongly coherent and strongly
compatible with $\Pi_{r}.$
\end{proposition}

\begin{remark}
Since $\left\Vert P\right\Vert _{d,r}\leq\left\Vert \overset{\vee}%
{P}\right\Vert _{d,r}$ for all $P\in\mathcal{P}_{d,r}^{n}$ it is obvious that
the sequence
\[
\left(  \left(  \mathcal{P}_{d,r}^{n},\left\Vert .\right\Vert _{d,r}\right)
,\left(  \mathcal{L}_{d,r}^{n},\left\Vert .\right\Vert _{d,r}\right)  \right)
_{n=1}^{\infty}%
\]
is also coherent and compatible with $\Pi_{r}.$
\end{remark}

\textbf{Acknowledgement. }The authors thank Geraldo Botelho for important suggestions.

\end{document}